\documentclass[11pt]{amsart}
\usepackage{amssymb,epsfig,newlfont,amsmath,amsfonts,mathscinet,hyperref}
\usepackage[all]{xy}

\usepackage[english]{babel}
\usepackage{graphicx}
\usepackage{multirow}

\newcommand{\Ind}{\mathrm{Ind}}

\newcommand{\colim}{\operatornamewithlimits{colim}}

\makeatletter
\def\revddots{\mathinner{\mkern1mu\raise\p@
\vbox{\kern7\p@\hbox{.}}\mkern2mu
\raise4\p@\hbox{.}\mkern1mu\raise7\p@\hbox{.}\mkern1mu}}
\makeatother

\begin{document}

\newtheorem{thm}{Theorem}[section]
\newtheorem{lem}[thm]{Lemma}
\newtheorem{cor}[thm]{Corollary}
\newtheorem{prop}[thm]{Proposition}
\newtheorem{conj}[thm]{Conjecture}

\newtheorem*{thm*}{Theorem}
\newtheorem*{assump}{Assumption}
\newtheorem*{assumpA}{Assumption A}
\newtheorem*{assumpA1}{Assumption A for arbitrary $r$}
\newtheorem*{assumpA2}{Assumption A for $r=q$}
\newtheorem*{assumpB}{Assumption B}
\newtheorem*{assumpC}{Assumption C}
\newtheorem*{conj*}{Conjecture}
\newtheorem*{conds*}{Conditions}
\newtheorem*{cond*}{Condition}

\theoremstyle{definition}
\newtheorem{defn}[thm]{Definition}
\newtheorem{rmk}[thm]{Remark}


\title{A Functorial Refinement of the Franke Filtration and the Jacquet--Langlands Correspondence for Spaces of Automorphic Forms}

\author{Neven Grbac}
\thanks{This work was supported by the Croatian Science Foundation under the project number HRZZ-IP-2022-10-4615.
This work was funded by the EU NextGeneration under the Juraj Dobrila University of Pula institutional research projects number IIP\_UNIPU\_010159 and IIP\_UNIPU\_010162.}
\address{Neven Grbac, Juraj Dobrila University of Pula, Zagreba\v{c}ka 30, HR-52100 Pula, Croatia}
\email{neven.grbac@unipu.hr}

\author{Harald Grobner}
\thanks{The second named author is also supported by the Principal Investigator project PAT4628923 of the Austrian Science Fund (FWF)}
\address{Harald Grobner, Faculty of Mathematics, University of Vienna, Oskar-Morgenstern-Platz 1, A-1090 Vienna, Austria}
\email{harald.grobner@univie.ac.at}

\begin{abstract}
The global Jacquet--Langlands correspondence is an instance of Langlands functoriality, namely
the expected lifting of the irreducible automorphic representations
of an inner form of the general linear group to the split form via the identity morphism of $L$-groups.
It is established, by the work of Badulescu, in the case of irreducible
components of the discrete spectrum. The purpose of this paper is
to extend this correspondence beyond the discrete spectrum.
To this end, the point of view of the Franke filtration of spaces of automorphic forms
is taken. In fact, our technical key ingredient is a functorial refinement of the Franke filtration, which allows us to establish the Jacquet--Langlands correspondence between consecutive quotients of this refined filtration on the general linear group and its inner form. As a result, our extended Jacquet--Langlands correspondence properly extends Badulescu's correspondence and contains the full functorial lift, predicted by Langlands functoriality.
\end{abstract}

\date{\today}

\maketitle

\markleft{\uppercase{Neven Grbac and Harald Grobner}}
\markright{\uppercase{The Jacquet--Langlands Correspondence for Spaces of Automorphic Forms}}

\tableofcontents

\section*{Introduction}
\label{sect:intro}

In the classical Langlands program for complex representations of reductive groups over global fields,
one of the fundamental predicted and expected features is Langlands functoriality, cf.~\cite{arthur:in-genesis-of-l-program},
\cite{arthur-gelbart:lectures}, \cite{cogdell:dual-gps-langlands-funct}.
Given an $L$-morphism between the $L$-groups of two reductive groups $G$ and $H$ over a number field,
Langlands functoriality predicts a natural lift of the irreducible automorphic representations
of their respective ad\`{e}lic groups via the global Langlands correspondence. This lift should hence be compatible with local versions
of functoriality. In particular, at unramified places the Satake parameters should be matched. The seed of these ideas lies in the Langlands' letter to Weil \cite{langlands:letter-to-weil} (appeared in print as part of \cite[Chap.~VI]{dumbaugh-schwermer:artin-book})
and \cite{langlands:functoriality-original-reference}.\\\\
The Jacquet--Langlands correspondence is one of the instances of Langlands functoriality.
It is the above lift for irreducible automorphic representations of inner forms of the
general linear group to their split form, that is, more precisely, from general linear groups $GL_m$ over division algebras $D$ over number fields $F$ to the associate split general linear group $GL_n$ over $F$ itself.
The general linear group and its inner forms share the same $L$-group, and the Jacquet--Langlands
correspondence is the Langlands functoriality associated with the identity map viewed as an $L$-morphism,
cf.~\cite{knapp-rogawski:appl-trace-formula-edinburgh}.\\\\
Historically, the Jacquet--Langlands correspondence was one of the first instances of Langlands functoriality
that was studied in some detail. Already the book of Jacquet and Langlands \cite{jacquet-langlands:book}
solved the case of the general linear group $GL(2)$ and the multiplicative group of a quaternion algebra,
both over a local and global field. The local and global Jacquet--Langlands correspondence was developed
over the following years and decades in a series of papers
\cite{flath:global-jl-gl3},
\cite{rogawski:local-jl-corr-with-div-alg},
\cite{deligne-kazhdan-vigneras:local-jl-corr},
\cite{flicker:orb-int-div-algs},
\cite{badulescu:local-jl-unitary},
culminating finally in the proof of the global Jacquet--Langlands correspondence for discrete spectrum irreducible automorphic representations in the work of Badulescu
\cite{badulescu:global-jl}, relying heavily on the study of the trace formula by Arthur and Clozel \cite{arthur-clozel:book}.
The proof of Badulescu assumes the technical condition that the division algebra splits at the archimedean places,
which was removed soon after by the paper of Badulescu and Renard \cite{badulescu-renard:global-jl-arcimedean}.
See also the very nice expository lecture notes of Badulescu on the proof of the global Jacquet--Langlands correspondence for the discrete spectrum.
\cite{badulescu:expository-proof-of-global-jl}.\\\\
The global Jacquet--Langlands correspondence in \cite{badulescu:global-jl} is compatible
with the local Jacquet--Langlands correpondence, which Badulescu extended from
the case of discrete series representations, achieved by Deligne--Kazhdan--Vign\'{e}ras in \cite{deligne-kazhdan-vigneras:local-jl-corr},
to the case of ($d$-compatible) irreducible unitary representations. This extended local correspondence is characterized by the equality, up to a sign, of characters of irreducible unitary representations on matching conjugacy classes of regular semisimple elements.\\\\
The natural question that can be raised at this point is whether Bedulescu's extended local correspondence for
irreducible unitary representations can be used to extend the global Jacquet--Langlands correspondence to
irreducible unitary automorphic representations, which are not in the discrete spectrum. It was already shown by Badulescu in \cite[Sect.~5.3]{badulescu:global-jl} that this is impossible. In \emph{loc.~cit.}, he provides counterexamples and concludes that the global Jacquet--Langlands correspondence cannot be extended to non-square-integrable irreducible unitary automorphic representations in a way, which is compatible with his local correspondence between unitary representations.\\\\
Hence, to extend the global Jacquet--Langlands correspondence beyond the discrete spectrum, a new point of view has to be taken, and this is the main theme of the present paper. In our new approach the focus is moved from the irreducible automorphic representations, viewed as abstract irreducible representations with (necessarily abstract) local irreducible components, to the internal representation theoretic structure of the spaces of all automorphic forms. In other words, instead of an abstract local/global correspondence between irreducible automorphic representations, we obtain a correspondence of specific automorphic realizations, which are not necessarily irreducible, but provide a functorial lift of spaces of concrete automorphic forms.\\\\
Let us explain this in more details. As a matter of fact, 
the first step in understanding the representation theoretical structure of the whole space of automorphic forms of a reductive group is the decomposition into a direct sum with respect to the {\it cuspidal support} \cite{langlands:letter-to-borel},
\cite{langlands:eis-ser-book}, \cite{moeglin-waldspurger:book}, \cite{franke-schwermer:decomp-aut-forms}, \cite{grobner-zunar:smooth-auto-forms}. In turn, the best way to investigate the structure of the individual summands in this decomposition is the Franke filtration \cite{franke:filtration}, \cite{franke-schwermer:decomp-aut-forms}, \cite{grbac:franke-filt-max-par}, \cite{grobner:franke-compositio}, \cite{grbac-grobner:franke-gln}, \cite{grbac:franke-sp4-memoirs}. It is a finite descending filtration, whose consecutive quotients are described in terms of representations parabolically induced from discrete spectrum automorphic representations of the Levi factors that satisfy a compatibility condition with the cuspidal support considered. Very roughly speaking, it provides an ordering of the Eisenstein series appearing, according to their (in large parts mysterious) analytic behaviour. Now, the crucial observation in order to establish a global Jacquet-Langlands correspondence beyond the discrete automorphic spectrum is to look at the various spaces of automorphic forms supported in a fixed cuspidal automorphic representation of a Levi factor, and to {\it transfer its fine internal structure given by the Franke filtration}, in particular, of its consecutive quotients.\\\\
We would like to emphasize that the Franke filtration is not
just a convenient way to establish the Jacquet--Langlands correspondence.
It is a natural and very useful approach to representation theoretic difficulties
in the internal structure of spaces of automorphic forms with many applications.
Franke \cite{franke:filtration} originally applied the filtration to prove a fundamental ``threefold'' of results, namely that every non-cuspidal
automorphic form is a derivative of an Eisenstein series or a residue thereof,
that the cohomology of arithmetic
congruence subgroups can be determined from the relative Lie algebra cohomology
of the space of automorphic forms (sometimes referred to as a conjecture of Borel and Harder), and a certain type of trace formula for
Hecke operators in full cohomology.
Generalizing a deep result of Clozel, Franke \cite{franke:filtration} and Franke--Schwermer \cite{franke-schwermer:decomp-aut-forms} used the Franke filtration to further prove the rationality of the decomposition along the parabolic and even the
cuspidal support in automorphic cohomology of the general linear group, a result which was finally extended to all inner forms of $GL(n)$ by Grobner--Raghuram in \cite{grobner-raghuram}.
Later, the explicit descriptions of the Franke filtration paved the way to constructions
of automorphic cohomology and cohomology of arithmetic groups
\cite{grbac-grobner:eis-coho-sp4},
\cite{grobner:franke-compositio},
\cite{grbac-schwermer:low-rank-existence},
\cite{grbac-schwermer:unitary-groups},
\cite{grbac-grobner:coho-slnZ},
and recently even to a new method for proving holomorphy of Eisenstein series
\cite{grbac:method},
\cite{grbac:method-sp4}.
See also \cite{grbac-schwermer:eis-coho-gl4-inner} and \cite{grobner-raghuram}
for an early discussion of the Jacquet--Langlands correspondence in cohomology.
Very recently, the explicit calculations of cohomology using the Franke filtration were
the crucial ingredients in the construction of non-trivial elements in a Bloch--Kato Selmer group
in accordance with the Bloch--Kato conjecture
\cite{mundy:franke-filt-eis-coho-g2}, \cite{mundy:franke-applied-g2-selmer-gps}.
In turn, the exciting work of Su \cite{su:coherent-coho-shimura-auto-forms}
reveals that the coherent cohomology of Shimura varieties is
isomorphic to the $(\mathfrak{p}_h,K_h)$-cohomology of the attached space of automorphic forms, opening a new door for investigations of the arithmetic of automorphic forms and representations appearing in coherent cohomology.
To cut a long story short, the Franke filtration is highly expected to continue to play a profound role in future developments. Hence, we expect that our construction of the Jacquet--Langlands correspondence for spaces of automorphic forms in terms of the Franke filtration can lead
to new explicit constructions in cohomology, as well as applications in
arithmetic and geometry.\\\\
In order to state here the main result, we introduce some notation and refer to the body of the
paper for more details.
Let $G$ be the general linear
group, defined over a number field $F$ with the ring of ad\`{e}les $\mathbb{A}$, and
let $G'$ be an inner form of $G$ given as the group of invertible elements
in the matrix algebra over a division algebra $D$ over $F$.
Let $\mathcal{A}_{\{P'\},\varphi(\pi')}$ be the space of automorphic forms on $G'(\mathbb{A})$
with the cuspidal support in the associate class $(\{P'\},\varphi(\pi'))$,
represented by a cuspidal automorphic representation $\pi'$ of the Levi factor of a parabolic subgroup $P'$ of $G'$.
It turns out, cf.~Step 0 in Sect.~\ref{sect:global-JL-beyond-discrete}, that the global Jaquet--Langlands
correspondence for the discrete automorphic spectrum \cite{badulescu:global-jl}
gives rise to the unique cuspidal support $(\{Q\},\varphi(\sigma))$ for $G(\mathbb{A})$
corresponding to the cuspidal support $(\{P'\},\varphi(\pi'))$ for $G'(\mathbb{A})$.
In other words, one should face the delicate task of establishing a global Jacquet--Langlands correspondence between the spaces $\mathcal{A}_{\{P'\},\varphi(\pi')}$ and $\mathcal{A}_{\{Q\},\varphi(\sigma)}$ of automorphic forms, that respects their internal representation theory.\\\\
To this end, one needs a compatible structural description of these two spaces of automorphic forms.
The main difficulty is that the Franke filtrations of spaces of automorphic forms on the general linear
group, on the one hand, and of its inner form, on the other hand, are not compatible, leading to the dictum that the Franke filtration as defined in \cite{franke:filtration}
is not functorial.\\\\
One of the main results of this paper is a functorial refinement of the
Franke filtration of spaces of automorphic forms on the general linear group. This functorial refinement of the Franke filtration and the Jacquet--Langlands correspondence for spaces
of automorphic forms are established in Theorem \ref{thm:franke-inner-form},
Theorem \ref{thm:franke-modified} and Theorem \ref{thm:main-result-JL-beyond},
which are stated in a combined and slightly simplified way in our theorem below, avoiding technical details. We point out that
the theorems in the body of the paper are more explicit and the established global Jacquet--Langlands correspondence can be
effectively calculated.

\begin{thm*}
Let $\mathcal{A}_{\{P'\},\varphi(\pi')}$, respectively, $\mathcal{A}_{\{Q\},\varphi(\sigma)}$, be the spaces of automorphic
forms on $G'(\mathbb{A})$, respectively, $G(\mathbb{A})$, with cuspidal support in the associate class
$(\{P'\},\varphi(\pi'))$, respectively, $(\{Q\},\varphi(\sigma))$, as above.
Then, there exists a Franke filtration
$$
\mathcal{A}_{\{P'\},\varphi(\pi')}=\mathcal{A}_{\{P'\},\varphi(\pi')}^0\supsetneqq
\mathcal{A}_{\{P'\},\varphi(\pi')}^1\supsetneqq \dots \supsetneqq \mathcal{A}_{\{P'\},\varphi(\pi')}^{\ell'}\supsetneqq \mathcal{A}_{\{P'\},\varphi(\pi')}^{\ell'+1}=\{0\}
$$
of the space $\mathcal{A}_{\{P'\},\varphi(\pi')}$, and
there exists a functorial refinement of the Franke filtration
$$
\mathcal{A}_{\{Q\},\varphi(\sigma)}=\mathcal{A}_{\{Q\},\varphi(\sigma)}^0\supsetneqq
\mathcal{A}_{\{Q\},\varphi(\sigma)}^1\supsetneqq \dots \supsetneqq \mathcal{A}_{\{Q\},\varphi(\sigma)}^{\ell}\supsetneqq \mathcal{A}_{\{Q\},\varphi(\sigma)}^{\ell+1}=\{0\}
$$
of the space $\mathcal{A}_{\{Q\},\varphi(\sigma)}$, such that there is an increasing function
$$
\phi:\{0,1,\dots ,\ell'\}\to\{0,1,\dots ,\ell\}
$$
and an extended global Jacquet--Langlands correspondence for spaces of automorphic forms, denoted by $\mathbf{G}$,
satisfying
$$
\mathbf{G}\left(\mathcal{A}_{\{P'\},\varphi(\pi')}^i\slash \mathcal{A}_{\{P'\},\varphi(\pi')}^{i+1}\right)\cong
\mathcal{A}_{\{Q\},\varphi(\sigma)}^{\phi(i)}\slash \mathcal{A}_{\{Q\},\varphi(\sigma)}^{\phi(i)+1}
$$
on the consecutive quotients of the filtration. These consecutive quotients of the Franke filtration are described in terms of
parabolic induction from discrete spectrum automorphic representations of Levi factors, and the established
Jacquet--Langlands correspondence $\mathbf{G}$ is compatible with the Jacquet--Langlands correspondence
for the discrete spectrum applied to the inducing data.
\end{thm*}
Let us elaborate a bit on the contents of the theorem. The crucial feature of our global Jacquet--Langlands
correspondence for spaces of automorphic forms is contained in the last sentence. The fact that the new
correspondence on the consecutive quotients of the filtration is obtained from the correspondence of the discrete spectra
of Levi factors in the inducing data actually means that our definition is precisely the desired extension
and generalization of the correspondence for the discrete automorphic spectrum of Badulescu \cite{badulescu:global-jl}.\\\\
The automorphic representations in our correspondence are the quotients of the
Franke filtration, which are most of the time not irreducible. The irreducible constituents of
these automorphic representations are unramified at almost all places. These unramified representations
are given by the Satake parameters arising via parabolic induction from the Satake
parameters of the unramified representation in the inducing data. Since our Jacquet--Langlands
correspondence is given by the Jacquet--Langlands correspondence for discrete spectrum applied
on the inducing data, the Satake parameters of the irreducible constituents match at almost
all places. This shows that our Jacquet--Langlands correspondence for spaces of automorphic
forms properly contains the functorial lift for all irreducible automorphic representations
predicted by Langlands functoriality as applied to the identity map between the $L$-groups of $G'$ and $G$.\\\\
On the other hand, our Jacquet--Langlands correspondence is not just a correspondence
which transfers abstract representations and verifies only {\it ex post} that they indeed admit automorphic realizations. Our Jacquet--Langlands correspondence is really a functorial lift of spaces of automorphic forms  themselves, that is, we work with the automorphic
realizations of the considered representations throughout. This fact is one of the crucial advantages
of the approach using the Franke filtration. The Jacquet--Langlands correspondence established in
this paper is the correspondence between quotients of spaces of automorphic forms that
are spanned by derivatives of degenerate Eisenstein series that can be explicitly described.
The fact that the representations in the correspondence are automorphic -- which is usually very difficult to establish! -- does hence not require any further
proof, as they are already realized in the spaces of automorphic forms.\\\\
As already mentioned, the theorems in the paper are more precise than the one stated here. First of all, the
Franke filtrations of $\mathcal{A}_{\{P'\},\varphi(\pi')}$ and $\mathcal{A}_{\{Q\},\varphi(\sigma)}$
are carefully chosen, so that the function $\phi$ can be explicitly calculated. A subtle point in the
construction is that in the filtration of $\mathcal{A}_{\{Q\},\varphi(\sigma)}$, as defined by Franke \cite{franke:filtration},
it can happen that one part of a quotient of the filtration is in the image of the Jacquet--Langlands
correspondence, while the other part of the same quotient is not in the image.
This cannot be avoided by choosing a different Franke filtration. Hence, we introduce
a functorial refinement of the Franke filtration in which
these parts of the same quotient are separated in two different quotients, and prove that it can always be achieved. Thus,
the filtration of $\mathcal{A}_{\{Q\},\varphi(\sigma)}$ that we define and use in the construction is strictly speaking not a
Franke filtration, and that is why we call it ``a functorial refinement of the Franke filtration'' in the theorem above.\\\\
The Franke filtration is defined in terms of certain groupoid of triples via a functor from the
groupoid to the category of representations of the given group. Our construction of the global
Jacquet--Langlands correspondence also establishes a correspondence between the groupoid of triples
for the general linear group and its inner form. We conclude this introduction by stating that another interesting property of our so established global Jacquet--Langlands correspondence is, that it is functorial in the sense that it commutes with the technically challenging functors
defining the Franke filtration.

\bigskip
\small
\noindent  {\it Acknowledgments:}
We are grateful to Ioan Baduescu for reading an earlier draft of the paper.
The final parts of the manuscript were written during our stay at the 
Erwin Schr\"{o}dinger Institute in Vienna. We sincerely appreciate its peaceful, yet stimulating atmosphere, as well as the warm hospitality of the staff.
\normalsize

\section{Preliminaries and notation}
\label{sect:preliminaries}

Let $F$ be an algebraic number field. Given a non-trivial place $v$ of $F$, we denote by $F_v$
the completion of $F$ at the place $v$. Let $\mathbb{A}$ be the ring of ad\`{e}les of $F$.
For any $F$-algebra $\mathcal{R}$, the matrix algebra of $r\times r$ matrices with entries in $\mathcal{R}$
is denoted by $M_r(\mathcal{R})$.

Let $D$ be a central division algebra over $F$ of degree $d>1$. We say that $D$ is split (or unramified) at a given place
$v$ of $F$, if
$$
D\otimes_F F_v \cong M_d(F_v).
$$
Otherwise, we say that $D$ is non-split (or ramified) at $v$. It is well-known that every $D$ splits at all but
finitely many places, and that the finite set $V_{\mathrm{ram}}$ of ramified places for $D$ is of even cardinality, cf.~\cite[Thm.~1.12]{platrap}.

Consider the matrix algebra $A'=M_n(D)$ of $n\times n$ matrices with entries in $D$. Given a place $v$ of $F$,
there exist a unique positive integer $r_v$ and a unique (up to isomorphism) division algebra $D_v$ of degree $d_v$ over $F_v$
such that $r_vd_v=nd$ and
$$
A_v':=A'\otimes_F F_v \cong M_{r_v}(D_v).
$$
Observe that if $v\not\in V_{\mathrm{ram}}$ is a place at which $D$ splits, then $d_v=1$ and $A_v'\cong M_{nd}(F_v)$.

The group of invertible elements in the algebra $A'$, viewed as an algebraic group defined over $F$, is denoted by $G_n'$,
and referred to as the general linear group over the division algebra $D$. Given a place $v$ of $F$, the group of $F_v$-rational
points of $G_n'$ is denoted by $G_n'(F_v)$. It is isomorphic to the group of invertible elements of the matrix algebra
$A_v'\cong M_{r_v}(D_v)$ in the notation as above, that is,
$$
G_n'(F_v)\cong GL_{r_v}(D_v).
$$
The group of ad\`{e}lic points of $G_n'$ is denoted by $G_n'(\mathbb{A})$.

In the case of $D=F$, i.e., the division algebra $D$ is a field, we have $d=1$, and the notation is modified by simply removing the prime sign.
Thus, we have $A=M_n(F)$, and the group of invertible elements of $A$, viewed as an algebraic group over $F$,
is denoted by $G_n$. It is the usual general linear group defined over the field $F$, so that, given a place $v$ of $F$,
$$
G_n(F_v)= GL_{n}(F_v),
$$
and the same holds for $G_n(\mathbb{A})$. The group $G_n'$ is an inner form of the general linear group $G_{nd}$.

We fix once and for all the choice of a minimal parabolic subgroup $P_0'$ (resp.~$P_0$) of $G_n'$ (resp.~$G_n$) defined over $F$, such that
they are the subgroups consisting of all upper-triangular matrices in some fixed basis of the underlying free $D$-modules (resp. $F$-vector spaces.)
Then, in both cases, the standard parabolic $F$-subgroups, i.e., those parabolic subgroups defined over $F$ that contain the fixed minimal one,
are in one-to-one correspondence with ordered partitions of $n$ into positive integers. Given such a partition $\underline{n}=(n_1,\dots ,n_k)$,
the corresponding standard parabolic subgroup $P_{\underline{n}}'$ (resp.~$P_{\underline{n}}$) consists of the block-upper-triangular
matrices in $G_n'$ (resp.~$G_n$) with blocks along the diagonal of sizes $n_1,\dots ,n_k$. The parabolic subgroup $P_{\underline{n}}'$ (resp.~$P_{\underline{n}}$)
exhibits the Levi decomposition with Levi factor
\begin{equation}\label{eq:M}
L_{\underline{n}}'\cong G_{n_1}'\times\dots\times G_{n_k}' \quad (\hbox{resp.~} L_{\underline{n}}\cong G_{n_1}\times\dots\times G_{n_k} ).
\end{equation}

Observe that, in both groups, two parabolic subgroups are associate if and only if the corresponding ordered partitions are
permutations of each other. In other words, the associate classes of parabolic subgroups are in one-to-one correspondence with
unordered partitions of $n$ into positive integers. All parabolic subgroups in this paper are standard, unless specified otherwise.

For all positive integers $n$, let $\nu$ denote both the character of $G_n'(\mathbb{A})$ given as the absolute value of the reduced norm and the character
of $G_n(\mathbb{A})$ given as the absolute value of the determinant.
For a given parabolic subgroup $P'$ of $G_n'$ (resp.~$P$ of $G_{n}$), let $X^\ast(P')$ (resp.~$X^\ast(P)$) be the $\mathbb{Z}$-module
of $F$-rational characters of $P'$ (resp.~$P$). Let $\check{\mathfrak{a}}_{P'}=X^\ast(P')\otimes_\mathbb{Z} \mathbb{R}$ and $\check{\mathfrak{a}}_{P}=X^\ast(P)\otimes_\mathbb{Z} \mathbb{R}$.
These are finite-dimensional real vector spaces, and we add subscript $\mathbb{C}$ to denote their complexifications.
If the parabolic subgroups $P'$ and $P$ both correspond to the ordered partition $\underline{n}=(n_1,\dots ,n_k)$ of $n$, then
these vector spaces are of dimension $k$. We fix the bases of $\check{\mathfrak{a}}_{P'}$ and $\check{\mathfrak{a}}_{P}$ and their complexifications in such a way that
the $k$-tuple $\underline{s}=(s_1,\dots ,s_k)$ of real or complex numbers corresponds to the character $\nu_{\underline{s}}$ given by the assignment
$$
(\ell_1, \ell_2,\dots ,\ell_k)\mapsto \nu (\ell_1)^{s_1} \nu(\ell_2)^{s_2} \dots  \nu(\ell_k)^{s_k},
$$
where $(\ell_1,\dots ,\ell_k)\in L'_{\underline{n}}(\mathbb{A})$ (resp.\ $(\ell_1,\dots ,\ell_k)\in L_{\underline{n}}(\mathbb{A})$),
with $\ell_j\in G_{n_j}'(\mathbb{A})$ (resp.\  $\ell_j\in G_{n_j}(\mathbb{A})$), according to \eqref{eq:M}.
Let $\check{\mathfrak{a}}_{P'}^{G_n'}$ and $\check{\mathfrak{a}}_{P}^{G_n}$
denote the subspaces consisting of characters that are trivial on the center. In the coordinates fixed above,
these are given by the condition
$$
n_1s_1+n_2s_2+\dots + n_ks_k=0,
$$
and are of dimension $k-1$, which is the corank of the parabolic subgroups $P'$ and $P$.
The closure of the positive Weyl chamber in $\check{\mathfrak{a}}_{P'}^{G_n'}$ and $\check{\mathfrak{a}}_{P}^{G_n}$ is given
by the condition
$$
s_1\geq s_2\geq \dots\geq s_k
$$
in coordinates as above.

For any given parabolic subgroup $P'$ of $G_n'$ (resp.~$P$ of $G_n$), there is a canonical inclusion of the
space $\check{\mathfrak{a}}_{P'}^{G_n'}$ into $\check{\mathfrak{a}}_{P_0'}^{G_n'}$
(resp.~$\check{\mathfrak{a}}_{P}^{G_n}$ into $\check{\mathfrak{a}}_{P_0}^{G_n}$), given
by restriction of characters. We denote this inclusion by the same letter $\iota'$ (resp.~$\iota$),
regardless of the parabolic subgroup $P'$ (resp.~$P$) to which it refers.
Explicitly in coordinates, if $P'$ and $P$ both correspond to the ordered partition $(n_1,\dots ,n_k)$ of $n$,
then
$$
\iota'(\underline{s})=(s_1,\dots ,s_1,s_2,\dots ,s_2,\dots ,s_k,\dots ,s_k)\in\check{\mathfrak{a}}_{P_0'}^{G_n'},
$$
$$
\iota(\underline{s})=(s_1,\dots ,s_1,s_2,\dots ,s_2,\dots ,s_k,\dots ,s_k)\in\check{\mathfrak{a}}_{P_0}^{G_n},
$$
where $s_j$ appears $n_j$ times on the right-hand sides, and
$\underline{s}=(s_1,\dots ,s_k)$ is viewed as an element in $\check{\mathfrak{a}}_{P'}^{G_n'}$ in the first line,
and as an element of $\check{\mathfrak{a}}_{P}^{G_n}$ in the second line.

A certain partial order on $\check{\mathfrak{a}}_{P_0'}^{G_n'}$ and $\check{\mathfrak{a}}_{P_0}^{G_n}$, denoted by $\succ$,
plays an important role in the definition of the Franke filtration below. In the general setting, it is defined
in \cite[p.~233]{franke:filtration}, and in the case of the general linear group in \cite[Eq.~(2.2)]{grbac-grobner:franke-gln}.
For the inner form of the general linear group the same inequalities apply. More precisely, for given elements
$\underline{s}=(s_1,\dots ,s_n)$ and $\underline{t}=(t_1,\dots ,t_n)$ of either $\check{\mathfrak{a}}_{P_0'}^{G_n'}$ or $\check{\mathfrak{a}}_{P_0}^{G_n}$,
we have that $\underline{s}\succ\underline{t}$ if and only if $\underline{s}\neq\underline{t}$ and the system of inequalities
\begin{align*}
  s_1 &\leq t_1 \\
  s_1+s_2 &\leq t_1+t_2 \\
   &\dots  \\
  s_1+s_2+\dots +s_{n-1} &\leq t_1+t_2+\dots +t_{n-1}
\end{align*}
is satisfied.

The Weyl group $W$ of $G_n'$ and $G_n$ is isomorphic to the symmetric group of $n$ letters. A permutation in $W$ acts on the maximal split torus
by permuting its factors. Given a parabolic subgroup $P'$ (resp.~$P$) of $G_n'$ (resp.~$G_n$), with the Levi factor $L'$ (resp.~$L$),
let $W_{L'}$ (resp.~$W_{L}$) denote the Weyl group of the Levi factor. Let $W^{P'}$ (resp.~$W^P$) denote the set of minimal coset
representatives for the right cosets in $W_{L'}\backslash W$ (resp.~$W_L\backslash W$). Let $W(L')$ (resp.~$W(L)$) denote the
set of $w\in W^{P'}$ (resp.~$w\in W^P$) such that the conjugate of $L'$ (resp.~$L$) by $w$ is a standard Levi subgroup.
If $P'$ and $P$ correspond to the same ordered partition $\underline{n}=(n_1,\dots ,n_k)$ of $n$, the $W(L')$ and $W(L)$ can
be both identified with the set of permutations of $k$ letters. A permutation in $W(L')$ and $W(L)$ acts on $L'$ and $L$
by permuting the diagonal blocks.

In the discussion above, a certain fixed choice of a maximal compact subgroup $K'$ (resp.~$K$) of $G_n'(\mathbb{A})$ (resp.~$G_n(\mathbb{A})$)
is implicitly assumed, as in \cite[Sect.~I.1.4]{moeglin-waldspurger:book}, or \cite[Sect.\ 9.2]{grobner:book}.
In particular, it is in good position with respect to the fixed choice $P_0'$ (resp.~$P_0$) of a minimal parabolic subgroup.

As often in the theory of automorphic forms,
we work with automorphic forms that are finite for the action of the fixed maximal compact subgroup,
as in \cite{borel-jacquet:corvallis}. Hence, the spaces of automorphic forms considered here are not representations of the full group
$G_n'(\mathbb{A})$ and $G_n(\mathbb{A})$. Nevertheless, we abuse the language, and refer to such spaces
as (automorphic) representations of $G_n'(\mathbb{A})$ and $G_n(\mathbb{A})$. The reader, who is interested in an extension of this theory, which in fact provides actual continuous representations of groups of ad\`{e}lic points, is referred to \cite{grobner:book} and \cite{grobner-zunar:smooth-auto-forms}.

\section{The global Jacquet--Langlands correspondence for the discrete spectrum}
\label{sect:disc-spec-gln}

In this section we recall the main global results of Badulescu \cite{badulescu:global-jl} and
Badulescu--Renard \cite{badulescu-renard:global-jl-arcimedean}, in which the global Jacquet--Langlands
correspondence for the discrete spectra of $G_n'(\mathbb{A})$ and $G_n(\mathbb{A})$ is
established. The former paper is written under the assumption that $D$ splits at
all archimedean places, and the latter removes that assumption. We avoid here all the
technical details required in the construction of the global Jacquet--Langlands correspondence,
and only state the results in the form relevant for this paper. In particular,
we do not explain the local Jacquet--Langlands correspondence for unitary representations,
the crucial local ingredient in the construction of Badulescu and Badulescu--Renard,
and refer the interested reader to \cite{badulescu:global-jl} and \cite{badulescu-renard:global-jl-arcimedean}.
However, for convenience of the reader familiar with the work of Badulescu and Renard, whenever
possible we use their notation and terminology of \cite{badulescu-renard:global-jl-arcimedean}.

We begin by recalling the classification of the discrete spectrum of $G_n(\mathbb{A})$ obtained
by M\oe glin and Waldspurger \cite{moeglin-waldspurger:gln}. Given a cuspidal automorphic representation
$\rho$ of $G_m(\mathbb{A})$, and a positive integer $k$ such that $mk=n$, the parabolically induced
representation
$$
{\Ind}_{P_{(m,m,\dots ,m)}(\mathbb{A})}^{G_n(\mathbb{A})}
\left(
\rho\nu^{\frac{k-1}{2}}\otimes\rho\nu^{\frac{k-3}{2}}\otimes \dots \otimes \rho\nu^{-\frac{k-1}{2}}
\right)
$$
has a unique irreducible quotient. Following \cite{badulescu:global-jl}, we denote that quotient by $MW(\rho,k)$.
The main result of \cite{moeglin-waldspurger:gln} says that $MW(\rho,k)$ is isomorphic to a residual automorphic representation
of $G_n(\mathbb{A})$ if $k>1$, and all residual representations of $G_n(\mathbb{A})$ arise in this way.
In the remaining case of $k=1$, we have $MW(\rho,1)=\rho$, which gives all the cuspidal automorphic representations
of $G_n(\mathbb{A})$.

A segment of cuspidal automorphic representations is a tensor product of the form
$$
\Delta = \Delta (\rho ;a,b) =\rho\nu^a\otimes\rho\nu^{a+1}\otimes \dots \otimes \rho\nu^b,
$$
where $\rho$ is a cuspidal automorphic representation of $G_m(\mathbb{A})$, and $a$ and $b$
are real numbers such that $\ell(\Delta)=b-a+1$ is a positive integer called the length of $\Delta$.
The segment $\Delta$ as above is the cuspidal support of the twisted discrete spectrum representation
$$
MW(\rho,b-a+1)\nu^{\frac{a+b}{2}}.
$$
Conversely, given a twisted discrete spectrum representation
$$
MW(\rho,k)\nu^s,
$$
where $\rho$ is a cuspidal automorphic representation of $G_m(\mathbb{A})$,
$k$ is a positive integer, and $s$ a real number, then its cuspidal support is
the segment
$$
\Delta\left(\rho ; s-\frac{k-1}{2},s+\frac{k-1}{2}\right).
$$
In other words, there is a bijection between all segments and all twisted discrete spectrum representations
of $G_n(\mathbb{A})$ for varying positive integer $n$.

The analogue of the M\oe glin--Waldspurger classification in the case of inner forms is obtained in the
course of establishing the global Jacquet--Langlands correspondence in \cite{badulescu:global-jl},
\cite{badulescu-renard:global-jl-arcimedean}.\footnote{The distinction of the residual representations
in the discrete spectrum of inner forms is the subject of the appendix by N.~Grbac in \cite{badulescu:global-jl}.}
We follow the exposition of \cite[Sect.~18]{badulescu-renard:global-jl-arcimedean}.
The global Jacquet--Langlands correspondence is an injective map, denoted by $\mathbf{G}$ as in \emph{loc.~cit.},
from the set of irreducible summands in the discrete spectrum of any of the groups $G_n'(\mathbb{A})$ with $n\geq 1$ into the set
of irreducible summands in the discrete spectrum of $G_{nd}(\mathbb{A})$, where $d$ is the degree of
the division algebra $D$ over $F$. The image $\mathbf{G}(\sigma')$ of
an irreducible summand $\sigma'$ of the discrete spectrum of $G_n'(\mathbb{A})$ is referred to as
the Jacquet--Langlands transfer of $\sigma'$.

Given a cuspidal automorphic representation $\rho$ of $G_m(\mathbb{A})$, there exists a unique positive
integer $k_\rho$ such that the representation $MW(\rho,k)$ is in the image of the Jacquet--Langlands
correspondence $\mathbf{G}$ if and only if $k_\rho$ divides $k$, cf.~\cite{badulescu:global-jl},
\cite{badulescu-renard:global-jl-arcimedean}.
This fact describes the image of the Jacquet--Langlands correspondence.
Moreover, $k_\rho$ divides $d$.

Let $\rho'$ be the irreducible representation in the discrete spectrum of $G_{m'}'(\mathbb{A})$ such that
$\mathbf{G}(\rho')=MW(\rho,k_\rho)$. In particular, we have $m'd=mk_\rho$. Then, $\rho'$ is
a cuspidal automorphic representation of $G_{m'}'(\mathbb{A})$, and all cuspidal automorphic representations
of $G_{m'}'(\mathbb{A})$ are of such form for an appropriate $\rho$, cf.~\cite{badulescu:global-jl},
\cite{badulescu-renard:global-jl-arcimedean}.

Given a positive integer $k'$, such that $m'k'=n$, the parabolically induced
representation
$$
{\Ind}_{P'_{(m',m',\dots ,m')}(\mathbb{A})}^{G_n'(\mathbb{A})}
\left(
\rho'\nu^{k_\rho\frac{k'-1}{2}}\otimes\rho'\nu^{k_\rho\frac{k'-3}{2}}\otimes \dots \otimes \rho'\nu^{-k_\rho\frac{k'-1}{2}}
\right)
$$
has a unique irreducible quotient denoted by $MW'(\rho' ,k')$. As in the case of $G_n(\mathbb{A})$, the representation
$MW'(\rho' ,k')$ is isomorphic to a residual automorphic representation of $G_n'(\mathbb{A})$ if $k'>1$, and all residual
representations of $G_n'(\mathbb{A})$ arise in this way. If $k'=1$, then $MW'(\rho',1 )=\rho'$ which gives all cuspidal
automorphic representations of $G_n'(\mathbb{A})$. See the appendix in \cite{badulescu:global-jl}.

Then, the Jacquet--Langlands correspondence is given explicitly by the assignment
$$
\mathbf{G}
\big(
MW'(\rho',k')
\big)=
MW(\rho,k),
$$
where $k=k_\rho k'$, and $\rho$ and $\rho'$ are related by the condition $\mathbf{G}(\rho')=MW(\rho,k_\rho)$ as above.

The notion of segments can be introduced in the case of inner forms. A segment for $G_n'$ is
a tensor product of cuspidal automorphic representations of the form
$$
\Delta' = \Delta' (\rho' ;a,b) =\rho'\nu^{k_\rho a}\otimes\rho'\nu^{k_\rho(a+1)}\otimes \dots \otimes \rho'\nu^{k_\rho b},
$$
where $a$ and $b$ are real numbers such that $b-a+1$ is a positive integer,
$\rho'$ is a cuspidal automorphic representation of $G_{m'}'(\mathbb{A})$, and
$\rho$ and $k_\rho$ are related to $\rho'$ by the condition $\mathbf{G}(\rho')=MW(\rho,k_\rho)$ as above. Then $\Delta'$ is the cuspidal
support of the representation
$$
MW'(\rho',b-a+1)\nu^{k_\rho\frac{a+b}{2}},
$$
and conversely, given a twisted discrete spectrum representation
$$
MW'(\rho',k)\nu^{k_\rho s},
$$
its cuspidal support is the segment
$$
\Delta'\left( \rho'; s-\frac{k-1}{2}, s+\frac{k-1}{2}\right).
$$
Hence, there is a bijection between all segments for $G_n'$ with varying $n\geq 1$,
and the twisted discrete spectrum representations of all $G_n'(\mathbb{A})$.

\section{The Franke filtration for inner forms of the general linear group}
\label{sect:filtr-defn}

The Franke filtration is originally defined in terms of certain triples \cite{franke:filtration}, \cite{franke-schwermer:decomp-aut-forms}.
Due to the results of \cite{moeglin-waldspurger:gln}, recalled in Sect.~\ref{sect:disc-spec-gln},
in the case of the general linear group
these triples can be combinatorially described in terms
of segments. The definition of the filtration in the case of the general linear
group is given in \cite[Sect.~2]{grbac-grobner:franke-gln}, and the
combinatorial description of the triples in \cite[Lemma 3.1]{grbac-grobner:franke-gln}.
The Franke filtration, as well as its combinatorial description in terms of segments,
can be extended to the case of inner forms of the general linear group. That is
the subject of this section. In addition, we carefully fix certain convenient choices in
the definition of the Franke filtration for the inner form,
so that the description of the Jacquet--Langlands correspondence below becomes feasible.

Let $G_n'$ be the inner form of the general linear group, as in Sect.~\ref{sect:preliminaries}.
We fix the cuspidal support as the associate class represented by a cuspidal automorphic representation
$$
\pi'\cong \rho_1'\nu^{s_1} \otimes \rho_2'\nu^{s_2} \otimes\dots \otimes \rho_l'\nu^{s_l},
$$
of the Levi factor $L'(\mathbb{A})$ of the parabolic subgroup $P'(\mathbb{A})$,
where $\rho_j'$ is a cuspidal automorphic representation of $G_{m_j'}'(\mathbb{A})$,
with $\sum_{j=1}^l m_j'=n$, and the exponents $s_1, \dots , s_l$ are real numbers.
Note that $P'$ is the parabolic subgroup corresponding to the ordered partition $(m_1',\dots ,m_l')$ of $n$.
Without loss of generality, we may and will assume that the representative is chosen in such
a way that $s_1\geq s_2\geq \dots \geq s_l$. Recall also that we assume $\sum_{j=1}^l m_j's_j = 0$.

Let $\mathcal{A}_{\{P'\},\varphi(\pi')}$ denote the space of automorphic forms
on $G_n'(\mathbb{A})$ with cuspidal support in the associate class $\varphi(\pi')$,
represented by $\pi'$, of cuspidal automorphic representations of the Levi factors
of parabolic subgroups in the associate class $\{P'\}$, represented by $P'$.
See \cite[Sect.~2]{grbac-grobner:franke-gln} for a precise definition
that applies to inner forms as well. Recall that two parabolic subgroups
are associate if their Levi factors are conjugate by an element of the Weyl group.
Two representations
of the Levi factors of associate parabolic subgroups are associate if they are
conjugate by the same Weyl group element.

In what follows, the cuspidal support $\{P'\}$ and $\varphi(\pi')$ is fixed. Hence,
we omit the reference to the cuspidal support in notation, except at the first mention
of a new object.

The triples for $\mathcal{A}_{\{P'\},\varphi(\pi')}$ are defined as follows.
Let $\mathcal{M}'=\mathcal{M}_{\{P'\},\varphi(\pi')}$ be the set of triples $(R',\Pi',\underline{z}')$ where
\begin{itemize}
  \item $R'$ is a parabolic subgroup of $G_n'$ containing an element of the associate class $\{P'\}$,
  \item $\Pi'$ is a unitary discrete spectrum representation of the Levi factor $L_{R'}(\mathbb{A})$,
  \item $\underline{z}'$ is an element of $\check{\mathfrak{a}}_{R'}^{G_n'}$ in the closure of the positive Weyl chamber for $R'$,
\end{itemize}
such that the cuspidal support of $\Pi'$ twisted by the character corresponding to $\underline{z}'$
is in the associate class $\varphi(\pi')$.
The structure of a groupoid
on $\mathcal{M}'$ is defined by taking the triples as objects,
and an isomorphism from a triple $(R',\Pi',\underline{z}')$ to $(S',\Sigma',\underline{\zeta}')$
is an element $w\in W(L_{R'})$ such that the conjugate action of $w$ transforms $L_{R'}$ into $L_{S'}$,
$\Pi'$ into $\Sigma'$, and $\underline{z}'$ into $\underline{\zeta}'$.
Recall that $w\in W(L_{R'})$ can be viewed as a permutation that acts on the Levi factor $L_{R'}$
by permuting the diagonal blocks, on the discrete spectrum representation $\Pi'$ by permuting factors in the
tensor product, and on elements of $\check{\mathfrak{a}}_{R'}^{G_n'}$ by permuting the coordinates.

Consider the set $\mathcal{S}'=\mathcal{S}_{\{P'\},\varphi(\pi')}$, consisting of all $\iota'(\underline{z}')$,
where $\underline{z}'$ ranges over the third entry of all triples in $\mathcal{M}'$, and $\iota'$ is the
inclusion introduced in Sect.~\ref{sect:preliminaries}.
It is a finite subset of $\check{\mathfrak{a}}_{P_0'}^{G_n'}$. The decisive ingredient
of the Franke filtration is a choice of a function $T'=T_{\{P'\},\varphi(\pi')}$ on the
set $\mathcal{S}'$, taking integer values, and compatible with the partial
order $\succ$ defined on $\check{\mathfrak{a}}_{P_0'}^{G_n'}$ in Sect.~\ref{sect:preliminaries}.
In other words, for $\iota'(\underline{z}')$ and $\iota'(\underline{\zeta}')$ in $\mathcal{S}'$,
we have
\begin{equation}\label{eq:T-inequality}
  T'(\iota'(\underline{z}'))>T'(\iota'(\underline{\zeta}')),
\end{equation}
whenever $\iota'(\underline{z}')\succ \iota'(\underline{\zeta}')$.

The Franke filtration of the space $\mathcal{A}_{\{P'\},\varphi(\pi')}$ is a descending filtration
$$
\dots \supseteq \mathcal{A}_{\{P'\},\varphi(\pi')}^i\supseteq \mathcal{A}_{\{P'\},\varphi(\pi')}^{i+1} \supseteq\dots ,
$$
indexed by integers, where the structure of the consecutive quotients of the filtration is described in terms of induced representations as
\begin{equation}\label{eq:filt-qts}
\mathcal{A}_{\{P'\},\varphi(\pi')}^i \slash \mathcal{A}_{\{P'\},\varphi(\pi')}^{i+1}
\cong
\colim_{\substack{(R',\Pi',\underline{z}')\in\mathcal{M}' \\ T'(\iota'(\underline{z}'))=i}}
\left(
{\Ind }_{R'(\mathbb{A})}^{G_n'(\mathbb{A})}\left(\Pi'\otimes\nu_{\underline{z}'}\right)
\otimes S(\check{\mathfrak{a}}_{R',\mathbb{C}}^{G_n'})
\right),
\end{equation}
where ${\Ind}$ stands for the normalized parabolic induction, $S(\check{\mathfrak{a}}_{R',\mathbb{C}}^{G_n'})$
denotes the symmetric algebra associated with the vector space $\check{\mathfrak{a}}_{R',\mathbb{C}}^{G_n'}$, and
the colimit is taken with respect to the functor $M'=M_{\{P'\},\varphi(\pi')}$ from the groupoid $\mathcal{M}'$ of triples to the category of
representations of $G_n'(\mathbb{A})$, as in \cite{franke:filtration} and \cite[Sect.~2]{grbac-grobner:franke-gln}.
The notion of a representation of $G_n'(\mathbb{A})$ is understood in the sense of Sect.~\ref{sect:preliminaries}.
The filtration is finite, as only finitely many quotients are non-trivial.

The functor $M'$ is defined on the objects of $\mathcal{M}'$ by
\begin{equation}\label{eq:functor}
  M'\big((R',\Pi',\underline{z}')\big) = {\Ind }_{R'(\mathbb{A})}^{G_n'(\mathbb{A})}\left(\Pi'\otimes\nu_{\underline{z}'}\right)
\otimes S(\check{\mathfrak{a}}_{R',\mathbb{C}}^{G_n'}).
\end{equation}
For an isomorphism $w\in W(L_{R'})$ from $(R',\Pi',\underline{z}')$ to $(S',\Sigma',\underline{\zeta}')$,
the action $M'(w)$ of the functor $M'$ is the intertwining operator between the two parabolically induced representations.
The precise definition is given in \cite[page 234]{franke:filtration}.

So far everything is in complete analogy to the case of the general linear group studied in \cite{grbac-grobner:franke-gln}.
At this point we specify a convenient choice of $T'=T_{\{P'\},\varphi(\pi')}$ to be used below in the construction
of the Jacqeut--Langlands correspondence of $\mathcal{A}_{\{P'\},\varphi(\pi')}$.

\begin{prop}\label{prop:partition-of-S-inner-form}
Let $\{P'\}$ and $\varphi(\pi')$ be a fixed cuspidal support as above. The finite set $\mathcal{S}'=\mathcal{S}_{\{P'\},\varphi(\pi')}$
admits a unique ordered partition
$$
\mathcal{S}'=\bigcup_{i=0}^{\ell'}\mathcal{S}'_i
$$
in disjoint non-empty subsets $\mathcal{S}'_i$ such that
\begin{enumerate}
\item[(i$\,'$)] the elements in the same subset $\mathcal{S}'_i$ are incomparable,
\item[(ii$\,'$)] for $0\leq i_0<i\leq \ell'$, elements $\iota'(\underline{z}')$ in $\mathcal{S}'_{i}$ and
$\iota'(\underline{z}'_0)$ in $\mathcal{S}'_{i_0}$ are either incomparable, or $\iota'(\underline{z}')\succ \iota'(\underline{z}'_0)$, and
\item[(iii$\,'$)] for every given element $\iota'(\underline{z}'_0)$ in $\mathcal{S}'_{i_0}$, and
for every integer $i$ such that $i_0<i\leq\ell'$,
there exists $\iota'(\underline{z}')$ in $\mathcal{S}'_i$ such that $\iota'(\underline{z}')\succ \iota'(\underline{z}'_0)$,
\end{enumerate}
in the partial order $\succ$ on
$\check{\mathfrak{a}}_{P_0'}^{G_n'}$ defined in Sect.~\ref{sect:preliminaries}. See Figure \ref{fig:partition-inner-form}.
\end{prop}

\begin{figure}
  \centering
  \includegraphics[scale=0.72]{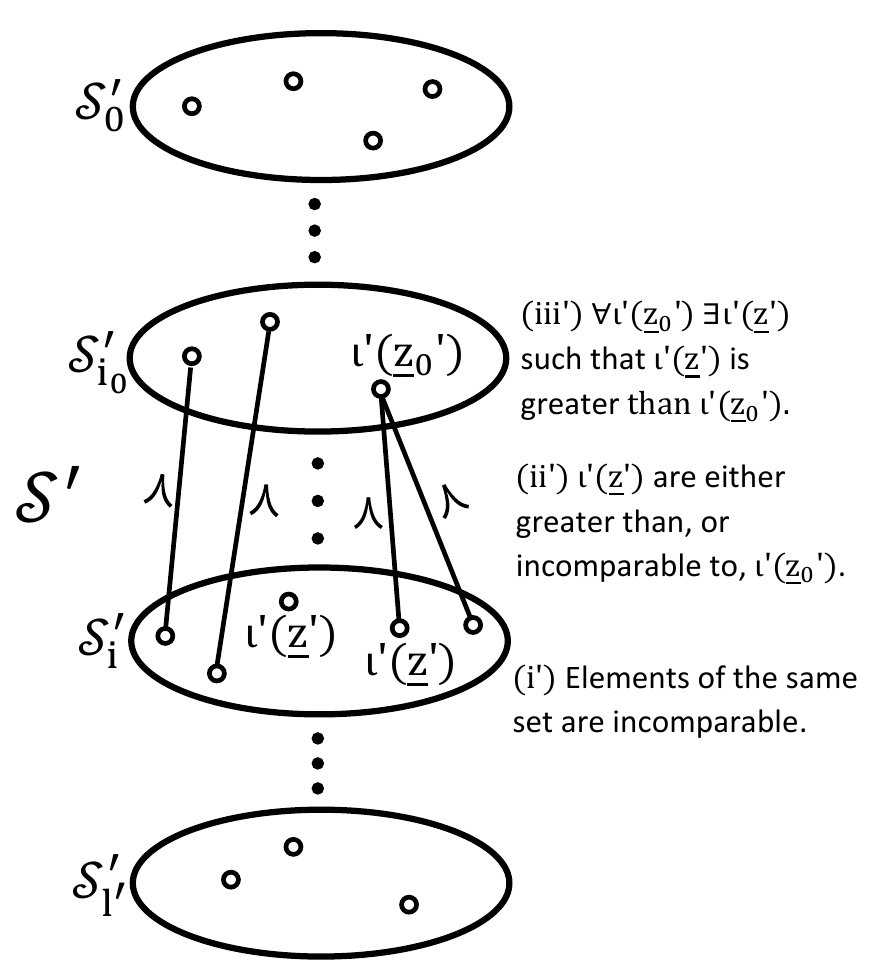}
  \caption{Partition of $\mathcal{S}'$ as in Proposition \ref{prop:partition-of-S-inner-form}}\label{fig:partition-inner-form}
\end{figure}

\begin{proof}
The construction of subsets $\mathcal{S}'_i$ is inductive, but we first obtain them in the reverse order than what is required.
It is a standard construction of a partition in antichains of a finite partially ordered set, and we recall it here
to fix the choices. Let
$\mathcal{S}^{\mathrm{aux}}_0$
be the set of all maximal elements in the partial order $\succ$ on $\mathcal{S}'$.
Since the set $\mathcal{S}'$ is non-empty and finite, $\mathcal{S}^{\mathrm{aux}}_0$ is non-empty.
Suppose that $i>0$ is an integer, and the non-empty subsets $\mathcal{S}^{\mathrm{aux}}_j$ are constructed for $0\leq j\leq i-1$.
If
$$
\bigcup_{j=0}^{i-1} \mathcal{S}^{\mathrm{aux}}_j = \mathcal{S}',
$$
the construction is finished and we have constructed the auxiliary partition of $\mathcal{S}'$.
Otherwise, a non-empty set $\mathcal{S}^{\mathrm{aux}}_i$ is defined as the set of all maximal elements in the
partial order $\succ$ on the non-empty finite set
$$
\mathcal{S}' \setminus \bigcup_{j=0}^{i-1} \mathcal{S}^{\mathrm{aux}}_j.
$$
In a finite number of steps, we obtain the auxiliary partition
$$
\mathcal{S}' = \bigcup_{i=0}^{\ell'} \mathcal{S}^{\mathrm{aux}}_i
$$
of the set $\mathcal{S}'$ in $\ell'+1$ subsets. The partition required in the proposition
is then simply defined by
$$
\mathcal{S}'_i=\mathcal{S}^{\mathrm{aux}}_{\ell'-i},
$$
for $0\leq i\leq \ell'$.
We claim that
this partition satisfies (i$\,'$)-(iii$\,'$) of the proposition.

By the construction using remaining maximal elements in each step, it is clear that (i$\,'$) holds,
because if two elements of the same subset $\mathcal{S}'_i$ were comparable, then one of them would not be
a maximal element in the appropriate set.

Suppose that (ii$\,'$) does not hold, that is, there exist
elements $\iota'(\underline{z}')$ in $\mathcal{S}'_i$ and
$\iota'(\underline{z}'_0)$ in $\mathcal{S}'_{i_0}$ such that $0\leq i_0<i\leq\ell'$ and
$\iota'(\underline{z}')\prec \iota'(\underline{z}'_0)$. But then, $\iota'(\underline{z}')$
is not a maximal element in the appropriate set, because $\iota'(\underline{z}'_0)$ is greater,
so we obtain a contradiction.

Finally, we first prove (iii$\,'$) for $i=i_0+1$. Let $\iota'(\underline{z}'_0)$
be any given element in $\mathcal{S}'_{i_0}$.
Suppose the contrary, that is,
there is no element $\iota'(\underline{z}')$ in $\mathcal{S}'_{i_0+1}$ such that
$\iota'(\underline{z}')\succ \iota'(\underline{z}'_0)$.
The other inequality is also not possible by (ii$\,'$), which is already proved.
Hence, $\iota'(\underline{z}'_0)$ would be incomparable with $\iota'(\underline{z}')$
for all elements $\iota'(\underline{z}')$ in $\mathcal{S}'_{i_0+1}$.
But this would mean that $\iota'(\underline{z}'_0)$ could have been chosen as a maximal element
already in the construction of $\mathcal{S}'_{i_0+1}$.
This is a contradiction with the fact that we choose all maximal elements in each step of the construction.
In general, if $i_0<i\leq\ell'$ is arbitrary, the claim follows step-by-step by looking at consecutive integers from $i_0$ all
the way up to $i$. For each pair of consecutive integers we already proved the claim, and invoking the
transitivity of the partial order $\succ$ implies (iii$\,'$).

It remains to prove uniqueness. Suppose that
$$
\mathcal{S}' = \bigcup_{i=0}^{\ell''} \mathcal{S}_i''
$$
is a different partition of the set $\mathcal{S}'$
satisfying (i$\,'$)-(iii$\,'$) of the proposition.
Let $i$ be the least integer such that $\mathcal{S}'_{\ell'-i}\neq\mathcal{S}''_{\ell''-i}$.
By the construction, $\mathcal{S}'_{\ell'-i}$ consists of all maximal elements in
the set
$$
\mathcal{S}'_{\mathrm{rem}}=\mathcal{S}'\setminus\bigcup_{j=\ell'-i+1}^{\ell'}\mathcal{S}'_j
=
\mathcal{S}'\setminus\bigcup_{j=\ell''-i+1}^{\ell''}\mathcal{S}''_j
$$
of remaining elements in the $i$-th step of the construction above.
Since the other partition satisfies (i$\,'$), all elements of $\mathcal{S}_{\ell''-i}''$
are incomparable, and according to (ii$\,'$) they are either greater than, or
incomparable to, all elements of $\mathcal{S}'_{\mathrm{rem}}\setminus \mathcal{S}''_{\ell''-i}$.
Hence, all elements of $\mathcal{S}_{\ell''-i}''$ are maximal in $\mathcal{S}'_{\mathrm{rem}}$,
so that $\mathcal{S}_{\ell''-i}''\subseteq \mathcal{S}_{\ell'-i}'$. The other inclusion follows from (iii$\,'$).
More precisely, if there exists $\iota'(\underline{z}')$ in $\mathcal{S}_{\ell'-i}'\setminus \mathcal{S}_{\ell''-i}''$,
then, in the second partition, $\iota'(\underline{z}')$ must belong to some $\mathcal{S}_{i_0}''$ with $i_0<\ell''-i$.
On the other hand, since $\mathcal{S}_{\ell''-i}''\subseteq \mathcal{S}_{\ell'-i}'$ and $\iota'(\underline{z}')\in \mathcal{S}_{\ell'-i}'$,
it follows from (i$\,'$) applied to $\mathcal{S}_{\ell'-i}'$ that $\iota'(\underline{z}')\in S_{i_0}''$ is incomparable to
all elements of $\mathcal{S}_{\ell''-i}''$. This is a contradiction with (iii$\,'$) for the second partition.
Hence, $\mathcal{S}_{\ell''-i}'' = \mathcal{S}_{\ell'-i}'$, and thus the two partitions are equal.
\end{proof}

\begin{thm}\label{thm:franke-inner-form}
Let $\mathcal{A}_{\{P'\},\varphi(\pi')}$ be the space of automorphic forms on $G_n'(\mathbb{A})$
with cuspidal support in the associate class represented by $\pi'$ as above. Then,
it can be arranged that the Franke filtration of
$\mathcal{A}_{\{P'\},\varphi(\pi')}$ takes the form
\begin{equation}\label{eq:1st-filt-in-thm-inner}
\mathcal{A}_{\{P'\},\varphi(\pi')}=\mathcal{A}_{\{P'\},\varphi(\pi')}^0\supsetneqq
\mathcal{A}_{\{P'\},\varphi(\pi')}^1\supsetneqq \dots \supsetneqq \mathcal{A}_{\{P'\},\varphi(\pi')}^{\ell'}\supsetneqq \mathcal{A}_{\{P'\},\varphi(\pi')}^{\ell'+1}=\{0\},
\end{equation}
where the consecutive quotients of the filtration are isomorphic to
\begin{equation}\label{eq:qts-in-thm-inner}
\mathcal{A}_{\{P'\},\varphi(\pi')}^i\slash \mathcal{A}_{\{P'\},\varphi(\pi')}^{i+1}\cong
\colim_{\substack{(R',\Pi',\underline{z}')\in\mathcal{M}' \\ \iota'(\underline{z}')\in \mathcal{S}'_i}}
\left(
{\Ind }_{R'(\mathbb{A})}^{G_n'(\mathbb{A})}\left(\Pi'\otimes\nu_{\underline{z}'}\right)
\otimes S(\check{\mathfrak{a}}_{R',\mathbb{C}}^{G_n'})
\right),
\end{equation}
for $i=0,1, \dots , \ell'$, and
$\mathcal{S}'_i$ are the subsets in the partition of $\mathcal{S}'=\mathcal{S}_{\{P'\},\varphi(\pi')}$ of Proposition \ref{prop:partition-of-S-inner-form}.
This form of the Franke filtration is the shortest
possible Franke filtration of $\mathcal{A}_{\{P'\},\varphi(\pi')}$.
\end{thm}

\begin{proof}
The form of the Franke filtration is determined by the choice of
function $T'=T_{\{P'\},\varphi(\pi')}$ on $\mathcal{S}'$
that satisfies condition \eqref{eq:T-inequality}.
We make the choice of $T'$ by assigning
$$
T'(\iota'(\underline{z}'))=i
$$
for every $\iota'(\underline{z}')$ in $\mathcal{S}'_i$,
where $i=0,1,\dots ,\ell'$.
This function satisfies condition \eqref{eq:T-inequality},
because of properties (i$\,'$) and (ii$\,'$) of subsets $\mathcal{S}'_i$ in Proposition \ref{prop:partition-of-S-inner-form}.
Since the preimage of $i=0,1,\dots ,\ell'$ under $T'$ is $\mathcal{S}'_i$,
the consecutive quotients of the Franke filtration \eqref{eq:1st-filt-in-thm-inner} obtained from this choice of $T'$ are
as in equation \eqref{eq:qts-in-thm-inner}.

The fact that the form \eqref{eq:1st-filt-in-thm-inner} of the Franke filtration is the shortest possible follows from
property (iii$\,'$) in Proposition \ref{prop:partition-of-S-inner-form} which implies that there is a chain of length $\ell'+1$ in
the partial order on $\mathcal{S}'$. Hence, the range of $T'$ must be at least of
cardinality $\ell'+1$, in order to satisfy condition \eqref{eq:T-inequality}, so that the length of the filtration is
at least $\ell'+1$, as claimed.
\end{proof}

\begin{rmk}
Fixing the form of the Franke filtration as in the theorem, or equivalently the fixed choice of function $T'=T_{\{P'\},\varphi(\pi')}$ as in the proof,
in the case of the inner form of the general linear group is required to make the Jacquet--Langlands correspondence explicit in Sect.~\ref{sect:global-JL-beyond-discrete}.
We remark that the choice of function $T'$ is exactly the choice that produces the
Franke filtration compatible with the ranks of Eisenstein series in the case of cuspidal support
of a residual representation of $G_n'(\mathbb{A})$, as in \cite[Sect.~4]{grbac-grobner:franke-gln}.
\end{rmk}

\section{The global Jacquet--Langlands correspondence beyond the discrete spectrum}
\label{sect:global-JL-beyond-discrete}

In this section we define the Jacquet--Langlands correspondence for spaces of automorphic forms.
It is defined in terms of the consecutive quotients of the
Franke filtration. The process of construction of the Jacquet--Langlands correspondence is divided in
several steps in which we study all the necessary
ingredients to formulate and prove our final result.

\subsection*{STEP 0: The cuspidal supports}

Step 0 in establishing the global Jacquet--Langlands correspondence is to
relate the cuspidal supports of the considered spaces of automorphic forms.
It is the step that detects which spaces of automorphic forms should be
related by the Jacquet--Langlands correspondence.

Let the cuspidal support for automorphic forms on the inner form $G_n'(\mathbb{A})$ be
the associate class represented by a cuspidal automorphic representation
$$
\pi'\cong \rho_1'\nu^{s_1} \otimes \rho_2'\nu^{s_2} \otimes\dots
\otimes \rho_l'\nu^{s_l},
$$
of the Levi factor $L'(\mathbb{A})$ of the parabolic subgroup $P'(\mathbb{A})$,
where $\rho_j'$ is a cuspidal automorphic representation of $G_{m_j'}'(\mathbb{A})$,
with $\sum_{j=1}^l m_j'=n$, and the exponents $s_1, \dots , s_l$ are real numbers,
as in Sect.~\ref{sect:filtr-defn}.
Recall that $P'$ is the parabolic subgroup corresponding to the ordered partition $(m_1',\dots ,m_l')$ of $n$,
we may assume that the representative is chosen in such
a way that $s_1\geq s_2\geq \dots \geq s_l$, and that $\sum_{j=1}^l m_j's_j = 0$.

We denote by
$$
\mathcal{E}(\pi')=\{s_1,s_2,\dots ,s_l\}
$$
the multiset of exponents appearing in the cuspidal support represented by $\pi'$. Given a cuspidal automorphic representation
$\rho'$ of $G'_{m'}(\mathbb{A})$, let $\mathcal{E}_{\rho'}(\pi')$ be the submultiset
of $\mathcal{E}(\pi')$ consisting of the exponents in $\pi'$ associated with $\rho'$.
Note that $\mathcal{E}_{\rho'}(\pi')$ is empty if $\rho'$ does not appear in $\pi'$.

The only natural way to define the Jacquet--Langlands transfer $\pi$ of the cuspidal support $\pi'$ is to apply
the global Jacquet--Langlands correspondence for discrete spectra
\begin{align*}
\pi=\mathbf{G}(\pi') &\cong \mathbf{G}(\rho_1')\nu^{s_1} \otimes \mathbf{G}(\rho_2')\nu^{s_2} \otimes\dots \otimes \mathbf{G}(\rho_l')\nu^{s_l} \\
&= MW(\rho_1,k_{\rho_1})\nu^{s_1} \otimes MW(\rho_2,k_{\rho_2})\nu^{s_2} \otimes\dots \otimes MW(\rho_l,k_{\rho_l})\nu^{s_l},
\end{align*}
as defined in Sect.~\ref{sect:disc-spec-gln}. Here $\rho_j$ is a cuspidal automorphic representation of $G_{m_j}(\mathbb{A})$
and $k_{\rho_j}$ is a positive integer with $m_j'd=m_jk_{\rho_j}$, which are determined by the condition $\mathbf{G}(\rho_j')=MW(\rho_j,k_{\rho_j})$.
Thus, $\pi$ is a (possibly residual) discrete spectrum representation of the Levi factor $M(\mathbb{A})$ of the parabolic
subgroup $P$ of $G_{nd}$ corresponding to the ordered partition $(m_1k_{\rho_1},m_2k_{\rho_2},\dots ,m_lk_{\rho_l})$ of $nd$.

However, since $\pi$ is not always cuspidal, we take its cuspidal support to serve as the cuspidal
support of the considered automorphic forms on $G_{nd}(\mathbb{A})$. Its representative can be given explicitly
as the cuspidal automorphic representation
\begin{align*}
  \sigma \cong & \rho_1\nu^{s_1+\frac{k_{\rho_1}-1}{2}} \otimes \rho_1\nu^{s_1+\frac{k_{\rho_1}-3}{2}} \otimes \dots \otimes \rho_1\nu^{s_1-\frac{k_{\rho_1}-1}{2}} \\
   & \otimes\rho_2\nu^{s_2+\frac{k_{\rho_2}-1}{2}} \otimes \rho_2\nu^{s_2+\frac{k_{\rho_2}-3}{2}} \otimes \dots \otimes \rho_2\nu^{s_2-\frac{k_{\rho_2}-1}{2}}  \otimes\dots \\
   & \dots\otimes \rho_l\nu^{s_l+\frac{k_{\rho_l}-1}{2}} \otimes \rho_l\nu^{s_l+\frac{k_{\rho_l}-3}{2}} \otimes \dots \otimes \rho_l\nu^{s_l-\frac{k_{\rho_l}-1}{2}}
\end{align*}
of the Levi factor of the parabolic subgroup $Q$ of $G_{nd}$ corresponding to the ordered partition
$$
(m_1,\dots ,m_1,m_2,\dots ,m_2,\dots ,m_l,\dots ,m_l)
$$
of $nd$, where $m_j$ appears $k_{\rho_j}$ times in the partition.
As in the case of inner form, we denote by $\mathcal{E}(\sigma)$ the multiset of exponents of $\sigma$,
and by $\mathcal{E}_\rho(\sigma)$ the submultiset of exponents associated with a cuspidal automorphic representation
$\rho$ of $G_m(\mathbb{A})$.

The following lemma shows that this construction is well-defined, that is, independent of the representative of the
associate class, and non-associate cuspidal supports for $G_n'(\mathbb{A})$ are assigned non-associate cuspidal supports for $G_{nd}(\mathbb{A})$.
Recall that associate cuspidal supports in our case are obtained from each other by a permutation of factors in the tensor product.

\begin{lem}\label{lem:cusp-supp-well-defnd}
In the construction of the representative $\sigma$ of the cuspidal support for $G_{nd}(\mathbb{A})$ from
the representative $\pi'$ of the cuspidal support for $G_n'(\mathbb{A})$ in the notation as above,
two representatives of cuspidal supports for $G_n'(\mathbb{A})$ are associate if and only if
the corresponding representatives of cuspidal supports for $G_{nd}(\mathbb{A})$ are associate.
\end{lem}

\begin{proof}
It is clear that associate cuspidal supports $\pi'$ give rise to associate $\sigma$, because they are associate by
the same permutation at the level of $\pi'$ and $\pi$. The more subtle fact is that non-associate $\pi'$ cannot
give rise to associate cuspidal supports $\sigma$.

The idea of the proof is to show that $\pi$ constructed from $\pi'$
is uniquely determined by $\sigma$, up to permutation of factors in the tensor product. By the construction of
$\sigma$, it contains a segment corresponding to each factor $MW(\rho_j,k_{\rho_j})\nu^{s_j}$ in the tensor product for $\pi$.
The problem is to prove that these segments cannot be rearranged in such a way that they produce a tensor product
not associate to $\pi$. More precisely, we must prove that any partition of $\sigma$ in segments of lengths $k_{\rho_j}$
is a permutation of the partition given by the factors in the tensor product for $\pi$.

The segments can be constructed only from factors with isomorphic $\rho_j$.
Hence, suppose that $\rho$ is a cuspidal automorphic representation isomorphic to some $\rho_j$. Then, the part of
$\sigma$ that can be used in the construction of the segment involving $\rho$ is of the form
\begin{align*}
  & \rho\nu^{t_1+\frac{k_{\rho}-1}{2}} \otimes \rho\nu^{t_1+\frac{k_{\rho}-3}{2}} \otimes \dots \otimes \rho\nu^{t_1-\frac{k_{\rho}-1}{2}} \otimes\dots \\
   & \dots\otimes \rho\nu^{t_r+\frac{k_{\rho}-1}{2}} \otimes \rho\nu^{t_r+\frac{k_{\rho}-3}{2}} \otimes \dots \otimes \rho\nu^{t_r-\frac{k_{\rho}-1}{2}},
\end{align*}
where $t_1,\dots ,t_r$ are the exponents in $\pi$ corresponding to all $\rho_j$ isomorphic to $\rho$.
In other words,
$$
\mathcal{E}_{\rho}(\sigma)=\left\{
t_1+\frac{k_{\rho}-1}{2},t_1+\frac{k_{\rho}-3}{2},\dots ,t_1-\frac{k_{\rho}-1}{2}, \dots ,
t_r+\frac{k_{\rho}-1}{2},t_r+\frac{k_{\rho}-3}{2},\dots ,t_r-\frac{k_{\rho}-1}{2}
\right\}
$$
is the multiset of exponents in $\sigma$ associated with $\rho$,
where $t_1,\dots, t_r$ form the multiset
$\mathcal{E}_{\rho'}(\pi')=\{t_1,\dots ,t_r\}$ obtained from $\pi'$, where $\rho'$ and $\rho$ are related as in the construction above.
The multiset $\mathcal{E}_{\rho}(\sigma)$ is obtained as the multiset sum of segments
$$
\Delta \left(\rho ; t_j-\frac{k_{\rho}-1}{2}, t_j+\frac{k_{\rho}-1}{2} \right),
$$
which are all of length $k_\rho$.
The goal is to prove that the same multiset $\mathcal{E}_{\rho}(\sigma)$ cannot be obtained as the multiset sum of segments all of
which are of length $k_\rho$ in a different way. This will imply that $\pi$ is uniquely determined by $\sigma$, up to permutation of factors in the tensor product, as required.

The above goal is achieved by the following algorithm that produces a unique partition of $\mathcal{E}_{\rho}(\sigma)$ in segments of length $k_\rho$.
Choose any maximal element of the multiset $\mathcal{E}_{\rho}(\sigma)$. It must be an element of a segment of length $k_\rho$.
This uniquely defines a segment, which can be removed from the multiset. In the remaining multiset, choose again a
maximal element, and remove the uniquely determined segment of length $k_\rho$ in which it is contained. This algorithm shows
that the partition of the multiset of exponents into segments of length $k_\rho$ can be constructed in a unique way, which is
the original partition from which the multiset is obtained. Thus, the lemma is proved.
\end{proof}

The lemma implies that, given a cuspidal support for $G_n'(\mathbb{A})$ represented by $\pi'$, the construction of $\sigma$
defines a unique cuspidal support for $G_{nd}(\mathbb{A})$ represented by $\sigma$, and the assignment of associate classes
$$
\varphi(\pi')\mapsto\varphi(\sigma)
$$
defines an injective map. Hence, as Step 0 towards the global Jacquet--Langlands correspondence of the spaces of
automorphic forms, we define the injective map, still denoted by $\mathbf{G}$, between the spaces of automorphic forms with a fixed cuspidal support
given by
$$
\mathbf{G}\left(
\mathcal{A}_{\{P'\},\varphi(\pi')}
\right)
=
\mathcal{A}_{\{Q\},\varphi(\sigma)}
$$
where $\sigma$ and $Q$ are determined by $\pi'$ and $P'$ as in the construction above.

It remains to characterize the spaces of automorphic forms on $G_{nd}(\mathbb{A})$ that are
in the image of the Jacquet--Langlands correspondence from $G_n'(\mathbb{A})$.
Let $\mathcal{A}_{\{Q_0\},\varphi(\sigma_0)}$
be the space of automorphic forms on $G_{nd}(\mathbb{A})$ represented by a cuspidal automorphic
representation $\sigma_0$ of the Levi factor of a parabolic subgroup $Q_0$. Consider the
multiset $\mathcal{E}(\sigma_0)$ of exponents of $\sigma_0$. It can be decomposed in a multiset
sum of $\mathcal{E}_\rho(\sigma_0)$ over cuspidal automorphic representations
$\rho$ of general linear groups. Then, it is clear from the construction above that
the space $\mathcal{A}_{\{Q_0\},\varphi(\sigma_0)}$ is in the image
of the Jacquet--Langlands correspondence if and only if $\mathcal{E}_\rho(\sigma_0)$ can be decomposed
in a multiset sum of segments of length $k_\rho$ for every $\rho$ such that $\mathcal{E}_\rho(\sigma_0)$ is non-empty.

\subsection*{STEP 1: The set of triples}

The next step in the construction of the global Jacquet--Langlands correspondence for spaces of automorphic forms
is to establish the Jacquet--Langlands transfer of triples in $\mathcal{M}'=\mathcal{M}_{\{P'\},\varphi(\pi')}$,
for a fixed cuspidal support represented by $\pi'$ as in Step 0. According to Step 0, the image of the transfer
is in the set $\mathcal{M}=\mathcal{M}_{\{Q\},\varphi(\sigma)}$ of triples for the transferred cuspidal support
represented by $\sigma$.

Let $(R',\Pi',\underline{z}')$ be a triple in $\mathcal{M}'$,
where
\begin{itemize}
  \item $R'$ is the parabolic subgroup of $G_n'$ corresponding to the ordered partition $(n_1,\dots ,n_r)$ of $n$,
  \item $\Pi'\cong\Pi_1'\otimes \dots\otimes \Pi_r'$ is the discrete spectrum representation of the Levi factor of $R'$,
  where $\Pi_j'$ is a unitary discrete spectrum representation of $G_{n_j}'(\mathbb{A})$,
  \item $\underline{z}'=(z_1,\dots ,z_r)$ is in the closure of the positive Weyl chamber in $\check{\mathfrak{a}}_{R'}^{G_n'}$,
  so that $z_j$ are real numbers such that $z_1\geq\dots\geq z_r$.
\end{itemize}
The Jacquet--Langlands transfer of triples is defined as the injective map, denoted again by $\mathbf{G}$, from the set of triples $\mathcal{M}'$
into the set of triples $\mathcal{M}$ given by the assignment
$$
\mathbf{G}\big(
(R',\Pi',\underline{z}')
\big)
=\big(\mathbf{G}(R'),\mathbf{G}(\Pi'),\mathbf{G}(\underline{z}')\big)
=(R,\Pi,\underline{z}),
$$
where $(R,\Pi,\underline{z})$ is the triple in $\mathcal{M}$ defined as
\begin{itemize}
  \item $R$ is the parabolic subgroup of $G_{nd}$ corresponding to the ordered partition $(n_1d,\dots ,n_rd)$ of $nd$,
  \item $\Pi\cong \Pi_1\otimes \dots\otimes\Pi_r$ is the discrete spectrum representation of the Levi factor of $R$,
  where $\Pi_j=\mathbf{G}(\Pi_j')$ is the global Jacquet--Langlands transfer of $\Pi_j'$ from $G_{n_j}'(\mathbb{A})$ to $G_{n_jd}(\mathbb{A})$ as in Sect.~\ref{sect:disc-spec-gln},
  \item $\underline{z}=(z_1,\dots ,z_r)$ is the same $r$-tuple as $\underline{z}'$, viewed as an element
  in the closure of the positive Weyl chamber in $\check{\mathfrak{a}}_{R}^{G_{nd}}$.
\end{itemize}
Observe that the last entry $\underline{z}'$ in the triples remains unchanged in the transfer, but is viewed
as an element of a different space.

According to \cite[Lemma 3.1]{grbac-grobner:franke-gln}, which applies also to the case of inner forms,
the sets of triples $\mathcal{M}'$ and $\mathcal{M}$ are in finite-to-one
correspondence with partitions of their cuspidal supports in segments.
The image of the Jacquet--Langlands correspondence $\mathbf{G}$ on the set of triples is described in the
following lemma in terms of partitions in segments.
We denote by $\mathcal{M}^{\mathbf{G}}=\mathcal{M}^{\mathbf{G}}_{\{Q\},\varphi(\sigma)}$ the set of triples obtained
as the Jacquet--Langlands transfer of the triples in $\mathcal{M}'$.

\begin{lem}\label{lem:image-triples}
The triple $(R,\Pi,\underline{z})$ in $\mathcal{M}=\mathcal{M}_{\{Q\},\varphi(\sigma)}$ is in the
image of the Jacquet--Langlands correspondence $\mathbf{G}$ of a triple in $\mathcal{M}'=\mathcal{M}_{\{P'\},\varphi(\pi')}$
if and only if it corresponds to the partition of multiset $\mathcal{E}(\sigma)$ of exponents appearing in the
tensor product for $\sigma$ into segments such that the lengths of all segments based on any given unitary cuspidal automorphic representation
$\rho$ are multiples of $k_\rho$.
\end{lem}

\begin{proof}
The idea of the proof is the same as in Lemma \ref{lem:cusp-supp-well-defnd}. It is clear that the image of a given triple
in $\mathcal{M}'$ is of the required form. The converse is again more subtle. Given a triple
in $\mathcal{M}'$ which corresponds to a partition of $\mathcal{E}(\sigma)$ as in the statement,
one should prove that it is obtained as the Jacquet--Langlands transfer of a triple in $\mathcal{M}'$.
But the length of a segment in the partition is a multiple of $k_\rho$ for $\rho$ on which the segment is based.
Hence, it is a union of $k$ segments of length $k_\rho$. It follows that decomposing each segment into a union of
segments of length $k_\rho$ gives rise to a decomposition of $\mathcal{E}(\sigma)$ into a multiset sum of
segments of length $k_\rho$. However, it is proved in Lemma \ref{lem:cusp-supp-well-defnd} that such a decomposition
is unique and corresponds to the decomposition of $\pi$ into tensor product, as required.
\end{proof}

\subsection*{STEP 2: The isomorphisms in the category of triples}

The purpose of this step is to show that the isomorphisms between triples are
preserved under the Jacquet--Langlands correspondence of triples.
Thus, the Jacquet--Langlands correspondence $\mathbf{G}$ can be viewed
as a (covariant) functor from the groupoid $\mathcal{M}'$ to the groupoid $\mathcal{M}$,
where its action on objects is given in Step 1, and now we define
the action on isomorphisms.

Recall that, given two (possibly equal) triples
$(R',\Pi',\underline{z}')$ and $(S',\Sigma',\underline{\zeta}')$
in $\mathcal{M}'$, an isomorphism from the former to the latter is any element $w'\in W(L_{R'})$
such that the conjugate $w'L_{R'}w'^{-1}$ of the Levi factor $L_{R'}$ of $R'$ is the Levi factor $L_{S'}$ of $S'$,
and the conjugates by $w'$ of $\Pi'$ and $\underline{z}'$ are $w'(\Pi')=\Sigma'$ and $w'(\underline{z}')=\underline{\zeta}'$, respectively.
If the parabolic subgroup $R'$ corresponds to the ordered partition
$\underline{n}=(n_1,\dots ,n_r)$ of $n$, then the elements of $W(L_{R'})$ can be identified with the set
of permutations of $r$ letters. In this identification, the permutation $w'\in W(L_{R'})$ acts on the triple
by permuting the diagonal blocks of the Levi factor of $R'$, the factors in the tensor product of $\Pi'$, and
the coordinates of $\underline{z}'$.

Let $(R,\Pi ,\underline{z})=\mathbf{G}\big((R',\Pi' ,\underline{z}')\big)$ and $(S,\Sigma ,\underline{\zeta})=\mathbf{G}\big((S',\Sigma' ,\underline{\zeta}')\big)$
be the Jacquet--Langlands transfers of $(R',\Pi' ,\underline{z}')$ and $(S',\Sigma' ,\underline{\zeta}')$, respectively.
These are triples in $\mathcal{M}$.
The isomorphisms from $(R,\Pi ,\underline{z})$ to $(S,\Sigma ,\underline{\zeta})$
are elements of $W(L_R)$ which conjugate $L_R$ to $L_S$, $\Pi$ to $\Sigma$, and $\underline{z}$ to $\underline{\zeta}$.
The parabolic subgroup $R$ corresponds to the ordered partition
$(n_1d,\dots ,n_rd)$ of $nd$. Hence, the elements of $W(L_R)$ can also be identified with permutations
of $r$ letters, under which the permutation $w\in W(L_R)$ acts on the triple by permuting the
diagonal blocks of the Levi factor $L_R$, the factors in the tensor product of $\Pi$ and the coordinates of $\underline{z}$.

Observe that by definition, the third entry $\underline{z}'$
of all the triples in $\mathcal{M}'$ belongs to the closure of the positive Weyl chamber
in $\check{\mathfrak{a}}_{R'}^{G_n'}$, which is given by the condition $z_1\geq \dots \geq z_r$ in coordinates.
Hence, an isomorphism $w'\in W(L_{R'})$ is a permutation of coordinates that preserves that condition.
More precisely, isomorphism may permute only the equal coordinates of $\underline{z}'$.
The same holds for the isomorphisms in $\mathcal{M}$.

Given $w'\in W(L_{R'})$, let $w=\mathbf{G}(w')\in W(L_R)$ be such that $w$ and $w'$ correspond to the same permutation
under identifications introduced above. The following lemma shows that isomorphisms are preserved under $\mathbf{G}$.

\begin{lem}\label{lem:morphisms-preserved}
In the notation as above, $w'\in W(L_{R'})$ is an isomorphism from $(R',\Pi',\underline{z}')$ to $(S',\Sigma',\underline{\zeta}')$
if and only if $w=\mathbf{G}(w')\in W(L_R)$ is an isomorphism from $(R,\Pi,\underline{z})=\mathbf{G}\big((R',\Pi',\underline{z}')\big)$
to $(S,\Sigma ,\underline{\zeta})=\mathbf{G}\big((S',\Sigma',\underline{\zeta}')\big)$.
In other words, the sets of isomorphisms are identified under $\mathbf{G}$.
\end{lem}

\begin{proof}
By definition of the Jacquet--Langlands correspondence on triples, it is clear that $\mathbf{G}$ commutes with the
action of $w'\in W(L_{R'})$ and $w\in W(L_R)$ that are identified with the same permutation. Thus, the claim follows.
\end{proof}

In the construction of the Jacquet--Langlands correspondence below, we also need the following simple lemma,
which assures that there are no isomorphisms between two triples in $\mathcal{M}$,
one of which is in the image of the Jacquet--Langlands correspondence, and the other is not.

\begin{lem}\label{lem:no-morphs-between-transfer-and-non}
Let $w\in W(R)$ be an isomorphism from the triple $(R,\Pi,\underline{z})$ to the triple $(S,\Sigma,\underline{\zeta})$
in $\mathcal{M}$. Then $(R,\Pi,\underline{z})$ is in the image of the Jacquet--Langlands correspondence
of triples if and only if $(S,\Sigma,\underline{\zeta})$ is in that image.
\end{lem}

\begin{proof}
Since isomorphisms $w\in W(R)$ are given as permutations of diagonal blocks, factors in tensor products and coordinates in the triple $(R,\Pi,\underline{z})$,
the lengths of the segments in the partition associated with $(R,\Pi,\underline{z})$ as in \cite[Lemma 3.1]{grbac-grobner:franke-gln},
remain unchanged under the action of $w$. On the other hand, the image of the Jacquet--Langlands correspondence of triples
is described in Lemma \ref{lem:image-triples} in terms of the lengths of these segments. Thus, the lemma follows.
\end{proof}

\subsection*{STEP 3: The partial order on the set of triples}

In this step we prove that the Jacquet--Langlands transfer of triples,
defined in Step 1, preserves the partial order required for the Franke filtration and introduced in Sect.~\ref{sect:preliminaries}.
At the first look, it seems there is nothing to prove, as the third entry of the triples remains unchanged. However, the
image in the minimal parabolic subgroup is what should be compared.
Given the third entry $\underline{z}'=(z_1,\dots ,z_r)$ of a triple in $\mathcal{M}'$, let
$$
\iota'(\underline{z}')=(\zeta_1,\dots ,\zeta_n)\in\check{\mathfrak{a}}_{P_0'}^{G_n'}
$$
denote the inclusion of $\underline{z}'$ into $\check{\mathfrak{a}}_{P_0'}^{G_n'}$.
Then, under the Jacquet--Langlands transfer, the inclusion of $\underline{z}=\mathbf{G}(\underline{z}')$ into $\check{\mathfrak{a}}_{P_0}^{G_{nd}}$
is clearly given as
$$
\iota(\mathbf{G}(\underline{z}'))=(\zeta_1,\dots ,\zeta_1,\zeta_2,\dots ,\zeta_2,\dots ,\zeta_n,\dots ,\zeta_n)
\in\check{\mathfrak{a}}_{P_0}^{G_{nd}},
$$
where each $\zeta_j$ appears $d$ times. The following lemma solves the issue of partial order under the Jacquet--Langlands correspondence.

\begin{lem}\label{lem:partial-order-preserved}
In the notation as above, the partial order $\succ$ recalled in Sect.~\ref{sect:preliminaries} is preserved under the Jacquet--Langlands correspondence, that is,
for any pair of triples in $\mathcal{M}'$ with the third entries $\underline{z}_1'$ and $\underline{z}_2'$, we have that
$$
\iota'(\underline{z}_1')\succ \iota'(\underline{z}_2')
$$
if and only if
$$
\iota(\mathbf{G}(\underline{z}_1'))\succ \iota (\mathbf{G}(\underline{z}_2')).
$$
\end{lem}

\begin{proof}
Since $\mathbf{G}$, as well as $\iota$ and $\iota'$, is a linear map, it is sufficient to prove that
$$
\iota'(\underline{z}')\succ 0 \in\check{\mathfrak{a}}_{P_0'}^{G_{n}'}
$$
if and only if
$$
\iota(\mathbf{G}(\underline{z}'))\succ 0 \in\check{\mathfrak{a}}_{P_0}^{G_{nd}}
$$
for every element $\underline{z}'$ in the appropriate space $\check{\mathfrak{a}}_{P'}^{G_{n}'}$.
Writing $\iota'(\underline{z}')$ and $\iota(\mathbf{G}(\underline{z}'))$ in terms of coordinates
as above, this claim is reduced to the equivalence of two systems of inequalities.
The first system consists of $n-1$ inequalities given as
$$
\zeta_1+\dots +\zeta_i\leq 0,
$$
where $i=1,\dots ,n-1$.
The second system consists of $nd-1$ inequalities given by
$$
d\zeta_1+\dots +d\zeta_{j-1}+c\zeta_j \leq 0,
$$
where $j=1,\dots ,n$ and $c=0,\dots ,d-1$, with $c\neq 0$ in the case of $j=1$.

Observe that all the inequalities of the first system are obtained from those of the second one.
More precisely the $i$th inequality of the first system is obtained from the inequality of
the second system given by $j=i+1$ and $c=0$ after dividing by $d$.

Conversely, suppose that the first system of inequalities holds. Let
$$
d\zeta_1+\dots +d\zeta_{j-1}+ c\zeta_{j}\leq 0
$$
be any inequality in the second system. We show that it is satisfied by an argument
depending on the sign of $\zeta_j$. If $\zeta_j\leq 0$, then
$$
d\zeta_1+\dots +d\zeta_{j-1}+ c\zeta_{j}
=d(\zeta_1+\dots +\zeta_{j-1})+ c\zeta_{j}\leq 0,
$$
because $c\zeta_j\leq 0$ and the sum in parentheses
$$
\zeta_1+\dots +\zeta_{j-1}\leq 0
$$
is one of the inequalities of the first system. Otherwise, if $\zeta_j>0$,
then
$$
d\zeta_1+\dots +d\zeta_{j-1}+ c\zeta_{j}
=d(\zeta_1+\dots +\zeta_{j-1}+\zeta_j)- (d-c)\zeta_{j}\leq 0,
$$
because $-(d-c)\zeta_{j}\leq 0$ and the sum in parentheses
$$
\zeta_1+\dots +\zeta_{j-1}+\zeta_j\leq 0
$$
is again one of the inequalities of the first system.
\end{proof}

\subsection*{STEP 4: A partition of the set $\mathcal{S}=\mathcal{S}_{\{Q\},\varphi(\sigma)}$}

We now define a convenient partition of the set $\mathcal{S}=\mathcal{S}_{\{Q\},\varphi(\sigma)}$ of the inclusions
$\iota(\underline{z})$ into $\check{\mathfrak{a}}_{P_0}^{G_{nd}}$ of all third entries $\underline{z}$ of triples in
$\mathcal{M}=\mathcal{M}_{\{Q\},\varphi(\sigma)}$. This partition is required to establish the Jacquet--Langlands correspondence.
It depends on the partition of the corresponding set $\mathcal{S}'=\mathcal{S}_{\{P'\},\varphi(\pi')}$ fixed in Proposition \ref{prop:partition-of-S-inner-form}.

Let $\mathcal{S}^{\mathbf{G}}$ denote the subset of $\mathcal{S}$
that consists of $\iota(\mathbf{G}(\underline{z}'))$, where $\underline{z}'$ ranges over the third entries of all
triples in $\mathcal{M}'=\mathcal{M}_{\{P'\},\varphi(\pi')}$. In other words, these are the inclusions $\iota$ of the third entries of all the
triples in $\mathcal{M}$ that are obtained as the Jacquet--Langlands transfer of the triples in
$\mathcal{M}'$.

We transfer the partition
$$
\mathcal{S}'=\bigcup_{i=0}^{\ell'} \mathcal{S}_i'
$$
of Proposition \ref{prop:partition-of-S-inner-form} to obtain the partition
$$
\mathcal{S}^{\mathbf{G}}=\bigcup_{i=0}^{\ell'} \mathcal{S}^{\mathbf{G}}_i,
$$
where
$$
\mathcal{S}^{\mathbf{G}}_i=\left\{
\iota(\mathbf{G}(\underline{z}'))\,:\, \iota'(\underline{z}')\in \mathcal{S}_i'
\right\}
$$
for $i=0,1,\dots ,\ell'$.
According to Lemma \ref{lem:partial-order-preserved},
the partial order is preserved under the Jacquet--Langlands correspondence,
so that the partition of $\mathcal{S}^{\mathbf{G}}$ satisfies the three properties (i$\,'$)-(iii$\,'$) of
Proposition \ref{prop:partition-of-S-inner-form}.

Let $\widetilde{\mathcal{S}}=\mathcal{S} \setminus \mathcal{S}^{\mathbf{G}}$, i.e.,
$\widetilde{\mathcal{S}}$ contains all elements of $\mathcal{S}$ that are not the third entries
of triples obtained as the Jacquet--Langlands transfer of triples in $\mathcal{M}'$.
The following simple lemma is required in the sequel.

\begin{lem}\label{lem:easy-partial-order}
Let $\iota(\underline{\zeta})$ be an element of $\widetilde{\mathcal{S}}$. Then,
\begin{itemize}
\item if $\iota(\underline{\zeta})\succ \iota(\underline{z})$ for some $\iota(\underline{z})\in \mathcal{S}_i^{\mathbf{G}}$,
then, for any given $\iota(\underline{z}_0)$ in $\mathcal{S}_i^{\mathbf{G}}$,
the element $\iota(\underline{\zeta})$ is either incomparable to $\iota(\underline{z}_0)$, or $\iota(\underline{\zeta}) \succ \iota(\underline{z}_0)$,
\item if $\iota(\underline{\zeta})\prec \iota(\underline{z})$ for some $\iota(\underline{z})\in \mathcal{S}_i^{\mathbf{G}}$,
then, for any given $\iota(\underline{z}_0)$ in $\mathcal{S}_i^{\mathbf{G}}$,
the element $\iota(\underline{\zeta})$ is either incomparable to $\iota(\underline{z}_0)$, or $\iota(\underline{\zeta}) \prec \iota(\underline{z}_0)$.
\end{itemize}
\end{lem}

\begin{proof}
We prove only the first claim, as the other is proved in the same way. Suppose the contrary, that is,
$\iota(\underline{\zeta})\succ \iota(\underline{z})$ for some $\iota(\underline{z})\in \mathcal{S}_i^{\mathbf{G}}$,
and there exists $\iota(\underline{z}_0)\in \mathcal{S}_i^{\mathbf{G}}$ such that $\iota(\underline{\zeta})\prec \iota(\underline{z}_0)$.
But then, the transitivity of the partial order implies that $\iota(\underline{z}_0)\succ \iota(\underline{z})$.
This contradicts property (i$\,'$) of Proposition \ref{prop:partition-of-S-inner-form}, which is transferred to the partition
of $\mathcal{S}^{\mathbf{G}}$ as explained above.
\end{proof}

The following proposition provides a convenient partition of $\widetilde{\mathcal{S}}$ with respect to the fixed partition of $\mathcal{S}^{\mathbf{G}}$,
arising from the partition of $\mathcal{S}'$ as above.
This is only the first coarse partition that will be refined below.

\begin{prop}\label{prop:partition-of-tilde-S}
The finite set $\widetilde{\mathcal{S}}$ admits a unique ordered partition
$$
\widetilde{\mathcal{S}}=\bigcup_{i=0}^{\ell'+1}\widetilde{\mathcal{S}}_i,
$$
in possibly empty disjoint subsets $\widetilde{\mathcal{S}}_i$ such that
\begin{enumerate}
\item[(i)] for $0\leq i_0< i\leq \ell'+1$, elements $\iota(\underline{z}_0)$ in $\mathcal{S}^{\mathbf{G}}_{i_0}$ and
$\iota(\underline{\zeta})$ in $\widetilde{\mathcal{S}}_i$ are either incomparable, or $\iota(\underline{\zeta})\succ\iota(\underline{z}_0)$,
\item[(ii)] for every given element $\iota(\underline{\zeta}_0)$ in $\widetilde{\mathcal{S}}_{i_0}$, and for every integer $i$ such that $i_0\leq i\leq \ell'$, there
exists an element $\iota(\underline{z})$ in $\mathcal{S}_i^{\mathbf{G}}$ such that $\iota(\underline{z})\succ\iota(\underline{\zeta}_0)$.
\end{enumerate}
Moreover, this unique partition of $\widetilde{S}$ satisfies
\begin{enumerate}
\item[(iii)] for $0\leq i_0< i\leq \ell'+1$, elements $\iota(\underline{\zeta}_0)$ in $\widetilde{\mathcal{S}}_{i_0}$ and
$\iota(\underline{\zeta})$ in $\widetilde{\mathcal{S}}_i$ are either incomparable, or $\iota(\underline{\zeta})\succ\iota(\underline{\zeta}_0)$
\end{enumerate}
in the partial order $\succ$ on $\check{\mathfrak{a}}_{P_0}^{G_{nd}}$ defined in Sect.~\ref{sect:preliminaries}.
See Figure \ref{fig:partition-split-form-1}.
\end{prop}

\begin{figure}
  \centering
  \includegraphics[scale=0.72]{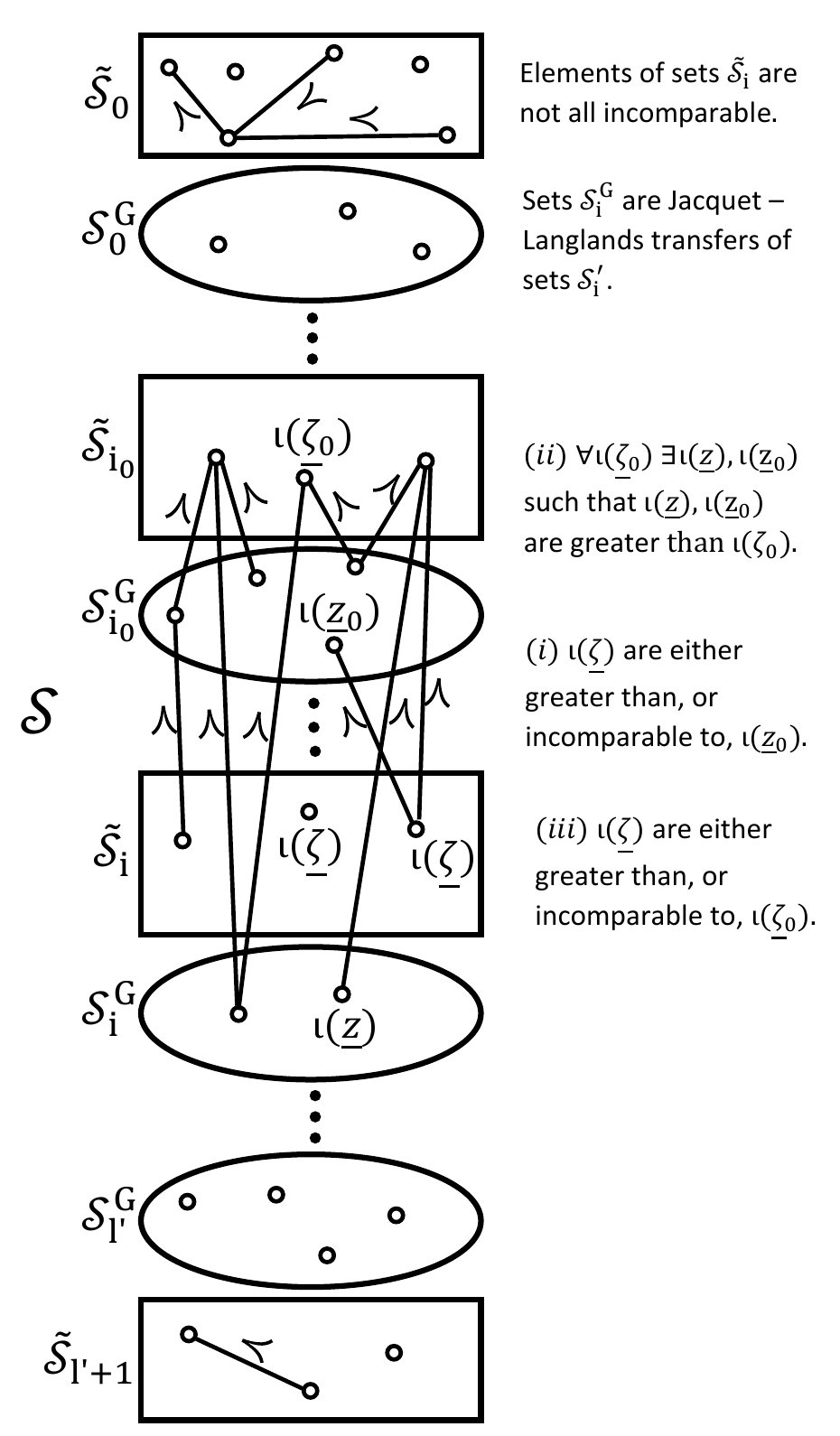}
  \caption{Partition of $\mathcal{S}$ as in Proposition \ref{prop:partition-of-tilde-S}}\label{fig:partition-split-form-1}
\end{figure}

\begin{proof}
The construction of subsets $\widetilde{\mathcal{S}}_i$ of the partition is downward inductive.
Let $\widetilde{\mathcal{S}}_{\ell'+1}$ be the (possibly empty) set of all $\iota(\underline{\zeta})$ in $\widetilde{\mathcal{S}}$
such that $\iota(\underline{\zeta})$ is either greater than, or incomparable to, every
element of $\mathcal{S}_{\ell'}^{\mathbf{G}}$. Suppose that $1\leq i<\ell'+1$ and that the subsets $\widetilde{\mathcal{S}}_j$ are constructed for $i+1\leq j\leq \ell'+1$.
Then, $\widetilde{\mathcal{S}}_i$ is defined as the (possibly empty) set of all $\iota(\underline{\zeta})$ in
$$
\widetilde{\mathcal{S}}\setminus\bigcup_{j=i+1}^{\ell'+1}\widetilde{\mathcal{S}}_j
$$
such that $\iota(\underline{\zeta})$ is either greater than, or incomparable to, every element of $\mathcal{S}_{i-1}^{\mathbf{G}}$.
The last (possibly empty) subset $\widetilde{\mathcal{S}}_{0}$ is defined as
$$
\widetilde{\mathcal{S}}_{0}=\widetilde{\mathcal{S}}\setminus \bigcup_{j=1}^{\ell'+1}\widetilde{S}_j,
$$
that is, it consists of all the remaining elements in $\widetilde{\mathcal{S}}$. We now prove that the constructed
partition satisfies properties (i)-(iii) of the proposition.

By the construction, property (i) holds for $i=i_0+1$. Suppose that $i>i_0+1$ and that (i) does not
hold, that is, there exist elements $\iota(\underline{z}_0)$ in $\mathcal{S}_{i_0}^{\mathbf{G}}$ and $\iota(\underline{\zeta})$ in $\widetilde{\mathcal{S}}_i$
such that $\iota(\underline{\zeta})\prec\iota(\underline{z}_0)$.
According to property (iii$\,'$) of Proposition \ref{prop:partition-of-S-inner-form}
applied to the partition of $\mathcal{S}^{\mathbf{G}}$ in $\mathcal{S}_i^{\mathbf{G}}$,
for any given $\iota(\underline{z}_0)$ in $\mathcal{S}_{i_0}^{\mathbf{G}}$,
there exists $\iota(\underline{z})$ in $\mathcal{S}_{i-1}^{\mathbf{G}}$ such that $\iota(\underline{z})\succ\iota(\underline{z}_0)$.
But then, the transitivity of the partial order implies that $\iota(\underline{\zeta})\prec\iota(\underline{z})$,
which is a contradiction with already proved property (i) in the case of $i=i_0+1$. Thus, property (i) is proved.

We now prove (ii) in the case of $i=i_0$. Suppose the contrary, that is, for the given element $\iota(\underline{\zeta}_0)$
in $\widetilde{\mathcal{S}}_{i_0}$, there is no element of $\mathcal{S}_{i_0}^{\mathbf{G}}$ greater than $\iota(\underline{\zeta}_0)$. In other words,
$\iota(\underline{\zeta}_0)$ is greater than, or incomparable to, every element of $\mathcal{S}_{i_0}^{\mathbf{G}}$. But then, by the construction,
$\iota(\underline{\zeta}_0)$ would have been chosen as an element of $\widetilde{\mathcal{S}}_i$ for some $i>i_0$, which is a contradiction.
Thus, property (ii) holds in the case of $i=i_0$. In the general case of $i_0< i\leq \ell'$, we prove (ii) using property (iii$\,'$) of
Proposition \ref{prop:partition-of-S-inner-form} and transitivity of the partial order.

To prove (iii), suppose the contrary, that is, for some $0\leq i_0<i\leq\ell'+1$, there exist $\iota(\underline{\zeta}_0)$ in $\widetilde{\mathcal{S}}_{i_0}$
and $\iota(\underline{\zeta})$ in $\widetilde{\mathcal{S}}_i$ such that $\iota(\underline{\zeta})\prec\iota(\underline{\zeta}_0)$. According to (ii),
which is already proved, there exists $\iota(\underline{z})$ in $\mathcal{S}_{i-1}^{\mathbf{G}}$ such that $\iota(\underline{z})\succ\iota(\underline{\zeta}_0)$.
Hence, by transitivity of the partial order, $\iota(\underline{\zeta})\prec\iota(\underline{z})$, which contradicts the construction
of $\widetilde{\mathcal{S}}_i$. Thus, property (iii) is proved.

It remains to prove uniqueness. Suppose there is another partition
$$
\widetilde{\mathcal{S}}=\bigcup_{i=0}^{\ell'+1}\widetilde{\widetilde{\mathcal{S}}}_i
$$
of $\widetilde{\mathcal{S}}$ satisfying (i)-(ii) of the proposition. Let $i$ be the greatest integer such that $\widetilde{\mathcal{S}}_{i}\neq\widetilde{\widetilde{\mathcal{S}}}_{i}$.
According to property (i), the elements of $\widetilde{\widetilde{\mathcal{S}}}_i$ are greater than, or incomparable to, every element of $\mathcal{S}_{i-1}^{\mathbf{G}}$.
It is then clear from the construction of $\widetilde{\mathcal{S}}_{i}$, that $\widetilde{\widetilde{\mathcal{S}}}_{i}\subseteq \widetilde{\mathcal{S}}_{i}$.
For the other inclusion, suppose there is an element $\iota(\underline{\zeta})$ in $\widetilde{\mathcal{S}}_i\setminus\widetilde{\widetilde{\mathcal{S}}}_i$.
In the second partition, the element $\iota(\underline{\zeta})$ belongs to some $\widetilde{\widetilde{\mathcal{S}}}_{i_0}$ with $i_0<i$.
According to property (ii), there exists $\iota(\underline{z})$ in $\mathcal{S}_{i-1}^{\mathbf{G}}$ such that $\iota(\underline{z})\succ\iota(\underline{\zeta})$.
This contradicts the construction of $\widetilde{\mathcal{S}}_i$, and thus $\widetilde{\mathcal{S}}_i=\widetilde{\widetilde{\mathcal{S}}}_i$ and the two partitions coincide.
\end{proof}

The partition of $\widetilde{\mathcal{S}}$ obtained in Proposition \ref{prop:partition-of-tilde-S} is refined using the same
procedure as in Proposition \ref{prop:partition-of-S-inner-form}. For each $\widetilde{\mathcal{S}}_i$, $i=0,1,\dots ,\ell'+1$,
there is a unique finite ordered partition
$$
\widetilde{\mathcal{S}}_i=\bigcup_{j=0}^{\ell_i} \widetilde{\mathcal{S}}_{i,j}
$$
satisfying properties (i$\,'$)-(iii$\,'$) of Proposition \ref{prop:partition-of-S-inner-form}.
The proof of this fact is the same as the proof of Proposition \ref{prop:partition-of-S-inner-form}.
The length of the partition of $\widetilde{\mathcal{S}}_i$ is $\ell_i +1$.
However, it is possible that $\widetilde{\mathcal{S}}_i$ is empty. In that case, we set $\ell_i=-1$, so that the length
of the empty partition of the empty set is $\ell_i+1=0$.

Combining the partition of $\mathcal{S}^{\mathbf{G}}$ and the refined partition of $\widetilde{\mathcal{S}}$,
we obtain the partition
$$
\mathcal{S}=\left(\bigcup_{i=0}^{\ell'}\mathcal{S}_i^{\mathbf{G}}\right)\cup\left(\bigcup_{i=0}^{\ell'+1}\bigcup_{j=0}^{\ell_i}\widetilde{\mathcal{S}}_{i,j}\right)
$$
of $\mathcal{S}$. It is of length
$$
L+1=\ell'+1+\sum_{i=0}^{\ell'+1}(\ell_i+1)=3+2\ell'+\sum_{i=0}^{\ell'+1}\ell_i.
$$
We fix the order of the subsets in this partition of $\mathcal{S}$ in the following way:
\begin{align*}
   & \widetilde{\mathcal{S}}_{0,0},\widetilde{\mathcal{S}}_{0,1},\dots ,\widetilde{\mathcal{S}}_{0,\ell_0}, \hskip 7cm \\
   & \mathcal{S}_0^{\mathbf{G}}, \\
   & \widetilde{\mathcal{S}}_{1,0},\widetilde{\mathcal{S}}_{1,1},\dots ,\widetilde{\mathcal{S}}_{1,\ell_1}, \\
(\clubsuit)\hskip 5.5cm & \mathcal{S}_1^{\mathbf{G}}, \\
   & \dots \\
   & \widetilde{\mathcal{S}}_{\ell',0},\widetilde{\mathcal{S}}_{\ell',1},\dots ,\widetilde{\mathcal{S}}_{\ell',\ell_{\ell'}}, \\
   & \mathcal{S}_{\ell'}^{\mathbf{G}}, \\
   & \widetilde{\mathcal{S}}_{\ell'+1,0},\widetilde{\mathcal{S}}_{\ell'+1,1},\dots ,\widetilde{\mathcal{S}}_{\ell'+1,\ell_{\ell'+1}}. \\
\end{align*}
See figure \ref{fig:partition-split-form-2}.
For later use, we introduce the convenient notation
$$
\mathcal{S}=\bigcup_{k=0}^{L}\mathcal{S}_k,
$$
for the ordered partition $(\clubsuit )$, where $L+1$ is its length. Here the subset $\mathcal{S}_k$ is the $k$-th subset of the ordered partition $(\clubsuit)$.
More precisely, for $0\leq i\leq \ell'$, let
$$
L_i=(\ell_0+1)+(\ell_1+1)+\dots +(\ell_{i}+1)+i=1+2i+\ell_0+\ell_1+\dots +\ell_i,
$$
and set $L_{-1}=-1$.
Then,
$$
\mathcal{S}_{L_i}=\mathcal{S}_i^{\mathbf{G}},
$$
for $0\leq i \leq \ell'$,
and
$$
\mathcal{S}_{L_i+j}=\widetilde{\mathcal{S}}_{i+1,j-1}
$$
for $-1\leq i\leq \ell'$ and $1\leq j\leq \ell_{i+1}+1$.
For the convenience of the reader, we also provide the converse formula
$$
\widetilde{\mathcal{S}}_{i,j}=\mathcal{S}_{L_{i-1}+j+1}
$$
for $0\leq i\leq \ell'+1$ and $0\leq j\leq \ell_i$.

\begin{figure}
  \centering
  \includegraphics[scale=0.72]{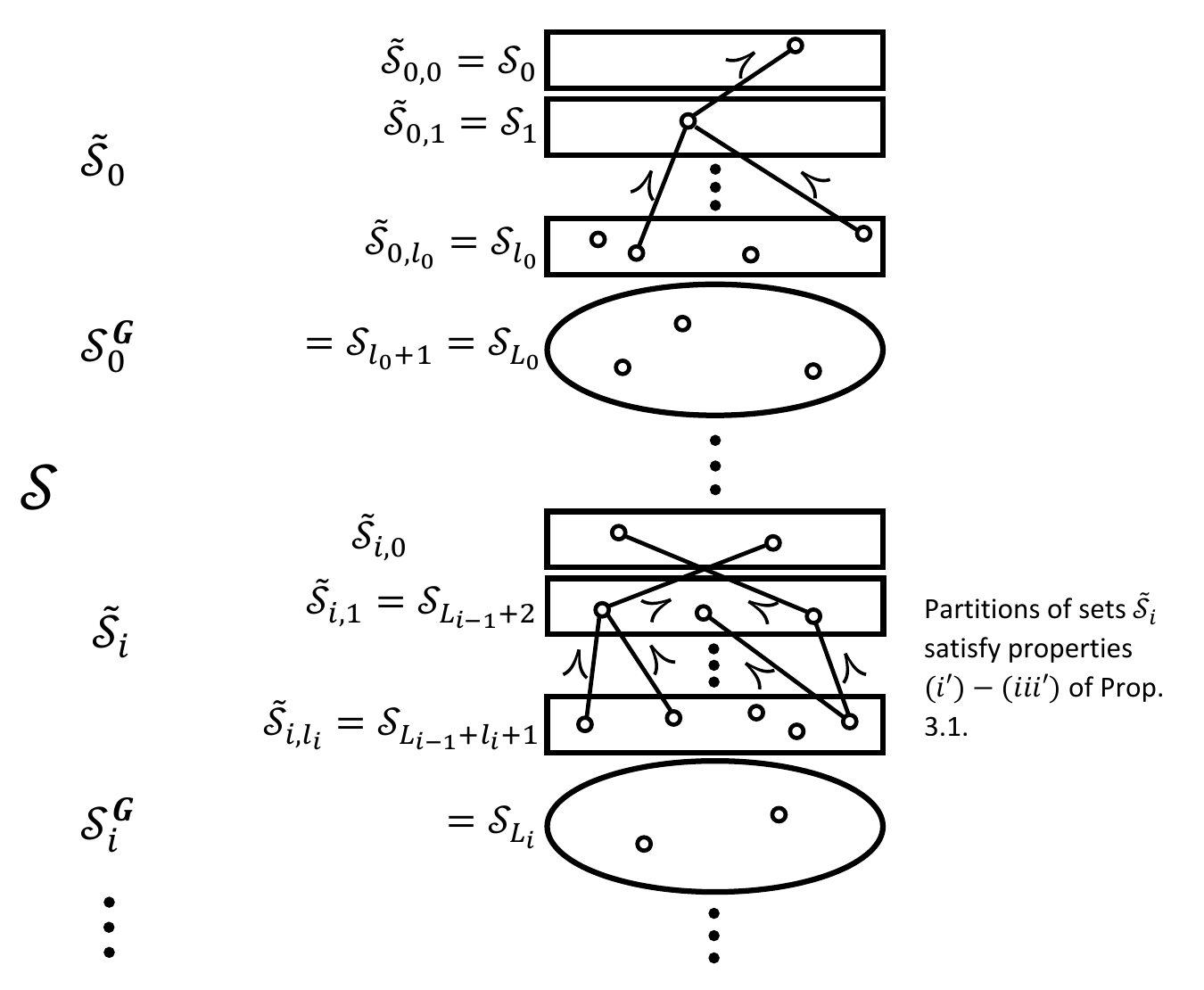}
  \caption{Partition $(\clubsuit)$ of $\mathcal{S}$ written also in the new notation}\label{fig:partition-split-form-2}
\end{figure}

\subsection*{STEP 5: A naive choice of the function $T=T_{\{Q\},\varphi(\sigma)}$}

In order to define explicitly the global Jacquet--Langlands correspondence,
we should carefully choose the function $T=T_{\{Q\},\varphi(\sigma)}$ in the definition
of the Franke filtration of $\mathcal{A}_{\{Q\},\varphi(\sigma)}$
compatibly with the fixed choice of $T'=T_{\{P'\},\varphi(\pi')}$ made in Theorem \ref{thm:franke-inner-form}.
This is achieved using the partition
$$
\mathcal{S}=\bigcup_{k=0}^{L} \mathcal{S}_k
$$
of the set $\mathcal{S}=\mathcal{S}_{\{Q\},\varphi(\sigma)}$
obtained at the end of Step 4, and denoted by $(\clubsuit)$, where $L+1$ is the length of the partition.

\begin{thm}\label{thm:franke-first-choice}
Let $\mathcal{A}_{\{Q\},\varphi(\sigma)}$ be the space of automorphic forms on $G_{nd}(\mathbb{A})$ with the cuspidal support
in the associate class represented by $\sigma$, obtained as in Step 1 from the cuspidal support $\{P'\},\varphi(\pi')$
of automorphic forms on $G_n'(\mathbb{A})$ represented by $\pi'$. Then, it can be arranged that the Franke filtration
of $\mathcal{A}_{\{Q\},\varphi(\sigma)}$ takes the form
\begin{equation}\label{eq:filt-in-thm-gln}
  \mathcal{A}_{\{Q\},\varphi(\sigma)}=\mathcal{A}_{\{Q\},\varphi(\sigma)}^0\supsetneqq
  \mathcal{A}_{\{Q\},\varphi(\sigma)}^1\supsetneqq \dots \supsetneqq \mathcal{A}_{\{Q\},\varphi(\sigma)}^{L}\supsetneqq \mathcal{A}_{\{Q\},\varphi(\sigma)}^{L+1}=\{0\},
\end{equation}
where the consecutive quotients of the filtration are isomorphic to
\begin{equation}\label{eq:qts-in-thm-gln}
\mathcal{A}_{\{Q\},\varphi(\sigma)}^k\slash \mathcal{A}_{\{Q\},\varphi(\sigma)}^{k+1}\cong
\colim_{\substack{(R,\Pi,\underline{z})\in\mathcal{M} \\ \iota(\underline{z})\in \mathcal{S}_{k}}}
\left(
{\Ind }_{R(\mathbb{A})}^{G_{nd}(\mathbb{A})}\left(\Pi\otimes\nu_{\underline{z}}\right)
\otimes S(\check{\mathfrak{a}}_{R,\mathbb{C}}^{G_{nd}})
\right),
\end{equation}
for $k=0,1, \dots , L$, and $\mathcal{S}_k$ are the subsets of the ordered partition $(\clubsuit )$ of
$\mathcal{S}$ obtained at the end of Step 4.
\end{thm}

\begin{proof}
The form of the Franke filtration is determined by the choice of the function $T=T_{\{Q\},\varphi(\sigma)}$
on $\mathcal{S}=\mathcal{S}_{\{Q\},\varphi(\sigma)}$ satisfying condition \eqref{eq:T-inequality}. The filtration in the theorem
is clearly obtained if $T$ is chosen compatibly with the partition of $\mathcal{S}$ as in Step 4. More precisely,
$$
T(\iota(\underline{z}))=k
$$
for every $\iota(\underline{z})$ in $\mathcal{S}_k$. This choice of $T$ satisfies \eqref{eq:T-inequality},
because of properties (i$\,'$) and (ii$\,'$) of Proposition \ref{prop:partition-of-S-inner-form}
and properties (i) and (ii) of Proposition \ref{prop:partition-of-tilde-S}.
\end{proof}

\subsection*{STEP $\mathbf{5\frac{1}{2}}$: A motivating example}

Observe that it is possible that there exist triples $(R,\Pi,\underline{z})$ and $(S,\Sigma ,\underline{\zeta})$
in $\mathcal{M}=\mathcal{M}_{\{Q\},\varphi(\sigma)}$ such that $(R,\Pi,\underline{z})$ is in the image of the Jacquet--Langlands
correspondence, but $(S,\Sigma ,\underline{\zeta})$ is not, and nevertheless $\iota(\underline{z})=\iota(\underline{\zeta})$.
The simplest example of this phenomenon is the following.\footnote{Note that $G_2$ in our notation stands for $GL_2$, and
\emph{not} the exceptional group of type $G_2$.}

Let $D$ be a quaternion algebra, so that $d=2$. Let
$\rho'$ be a cuspidal automorphic representation of $G_2'(\mathbb{A})$ such that its Jacquet--Langlands transfer
is the representation
$$
\mathbf{G}(\rho')=MW(\rho ,2)
$$
of $G_4(\mathbb{A})$, and thus, $\rho$ is a cuspidal automorphic representation of $G_2(\mathbb{A})$ and $k_\rho =2$.
Let $\tau'$ be a representation of $G_1'(\mathbb{A})$ of dimension more than one.
It is well-known, already from \cite{jacquet-langlands:book}
that the Jacquet--Langlands transfer of $\tau'$
is a cuspidal automorphic representation $\tau$ of $G_2(\mathbb{A})$, so that $k_\tau=1$.
Consider the cuspidal support represented by the cuspidal automorphic representation
$$
\pi'\cong\tau'\nu^{1/2}\otimes\rho' \otimes \tau'\nu^{-1/2}
$$
of the Levi factor of the parabolic subgroup $P'$ of $G_4'$ corresponding to the ordered partition $(1,2,1)$.
Its Jacquet--Langlands transfer is the representation
$$
\pi\cong
\tau\nu^{1/2}\otimes MW(\rho,2) \otimes\tau\nu^{-1/2}
$$
of the Levi factor of the parabolic subgroup $P$ of $G_8$ corresponding to the ordered partition $(2,4,2)$.
Hence, the cuspidal support associated with the cuspidal support $\pi'$ is represented by the cuspidal automorphic representation
$$
\sigma\cong
\tau\nu^{1/2}\otimes\rho\nu^{1/2}\otimes \rho\nu^{-1/2}\otimes\tau\nu^{-1/2}
$$
of the Levi factor of the parabolic subgroup $Q$ of $G_8$ corresponding to the partition $(2,2,2,2)$.
Let
$$
(R,\Pi,\underline{z})=(P,\tau\otimes MW(\rho,2)\otimes\tau,(1/2,0,-1/2))
$$
be the triple in $\mathcal{M}$ given by $\pi$. It is the Jacquet--Langlands transfer of
the triple
$$
(P',\tau'\otimes\rho'\otimes\tau',(1/2,0,-1/2))\in \mathcal{M}'=\mathcal{M}_{\{P'\},\varphi(\pi')},
$$
and we have
$$
\iota (\underline{z})
=\iota \big( (1/2,0,-1/2) \big) = (1/2,1/2,0,0,0,0,-1/2,-1/2).
$$
On the other hand, consider the triple
$$
(S,\Sigma ,\underline{\zeta})
=(S,\rho\otimes MW(\tau ,2)\otimes\rho,(1/2,0,-1/2)),
$$
which is also in $\mathcal{M}$. It is not in the image of the
Jacquet--Langlands correspondence, because $k_\rho=2$ so that the segments associated to
$\rho$ must be of even length according to Lemma \ref{lem:image-triples}, and here
they are of length one. However, we have again
$$
\iota (\underline{\zeta})=
\iota \big((1/2,0,-1/2)\big)=(1/2,1/2,0,0,0,0,-1/2,-1/2),
$$
that is, $\iota(\underline{z})=\iota(\underline{\zeta})$.

\subsection*{STEP 6: A functorial refinement of the Franke filtration}

The upshot of the motivating example given in the previous (half-)step is that
we cannot use the function $T_{\{Q\},\varphi(\sigma)}$ as introduced by Franke in \cite{franke:filtration}
to establish the Jacquet--Langlands correspondence for spaces of automorphic forms.
The reason is that in Franke's definition of the filtration, as illustrated by the example,
it is possible that the triples in the image of the Jacquet--Langlands correspondence and those that are not in the image
have the same $\iota(\underline{z})$. Thus, such triples cannot be assigned different
values of the function $T_{\{Q\},\varphi(\sigma)}$, so that they cannot be divided into two different consecutive
quotients of the filtration, which is required in the construction of the Jacquet--Langlands correspondence
for spaces of automorphic forms. In other words, the Franke filtration as defined in \cite{franke:filtration}
and constructed in Theorem \ref{thm:franke-first-choice} is not functorial.

The key new insight required for the construction of the Jacquet--Langlands correspondence is a functorial refinement of
the Franke filtration of Theorem \ref{thm:franke-first-choice}. The refinement should allow two different
triples with the same $\iota(\underline{z})$ to contribute to different quotients of the filtration. Although this is not the Franke filtration as
defined in \cite{franke:filtration}, we now prove that the required functorial refinement can always be achieved in our context.
This is the subject of Theorem \ref{thm:franke-modified} below.

In order to be able to construct the required refinement of the Franke filtration, we must distinguish between equal elements
$\iota(\underline{z})$ in $\mathcal{S}^\mathbf{G}$ depending on whether they arise only from the triples that are
in the image of the Jacquet--Langlands transfer, or not. For $0\leq i\leq \ell'$, let $\epsilon_i\in\{0,1\}$ be the
indicator of whether there exist triples in $\mathcal{M}=\mathcal{M}_{\{Q\},\varphi(\sigma)}$ such that $\iota(\underline{z})$ is in $\mathcal{S}_i^{\mathbf{G}}$,
but they are not in the image of the Jacquet--Langlands correspondence for triples. If that is the case, we set $\epsilon_i=1$, and
otherwise $\epsilon_i=0$.

\begin{figure}
  \centering
  \includegraphics[scale=0.72]{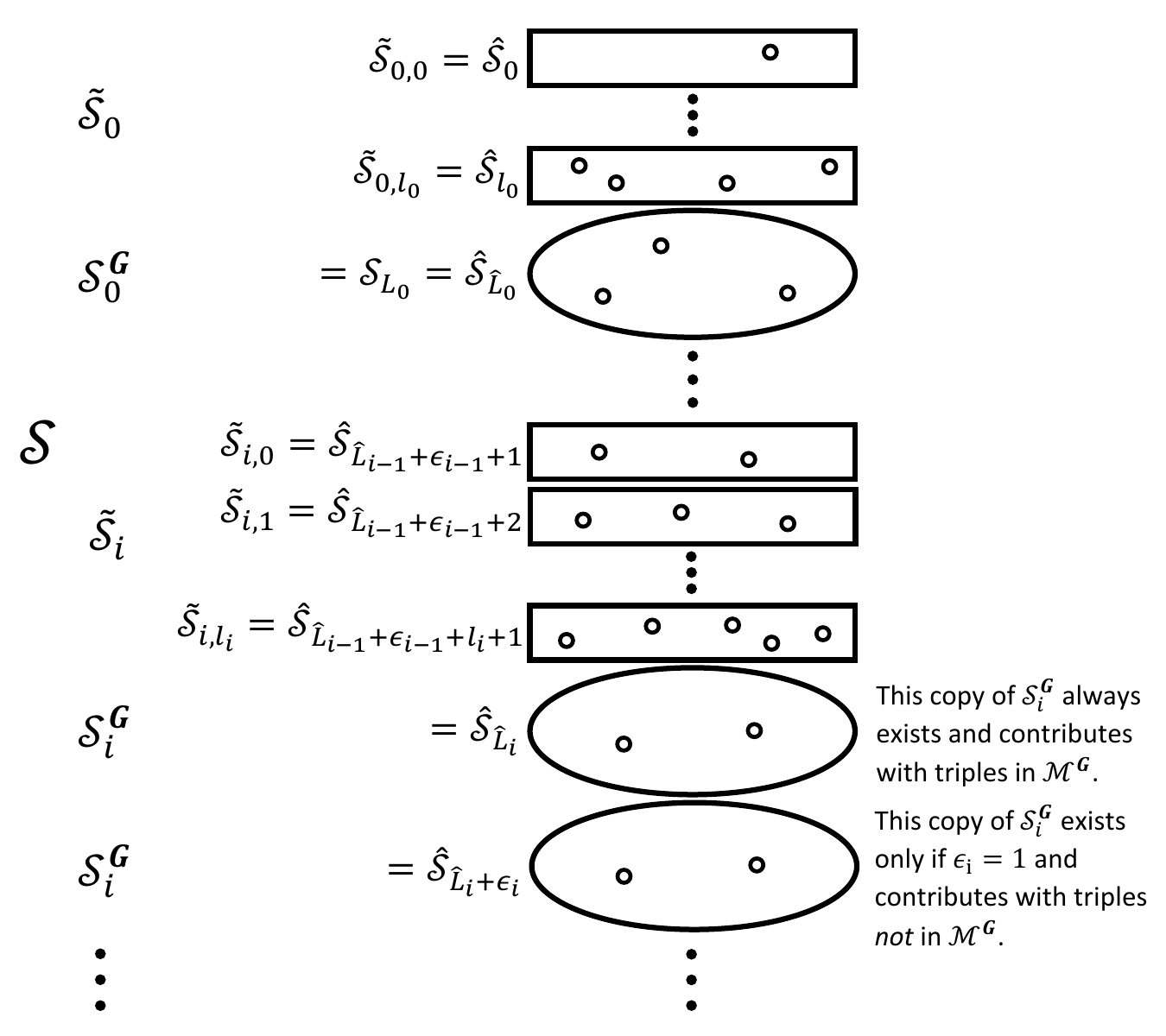}
  \caption{Functorially refined partition of $\mathcal{S}$ as required in Theorem \ref{thm:franke-modified}}\label{fig:partition-split-form-3}
\end{figure}

Before stating the theorem on the refined filtration, we introduce some notation.
In the ordered partition $(\clubsuit )$ of $\mathcal{S}=\mathcal{S}_{\{Q\},\varphi(\sigma)}$ in subsets $\mathcal{S}_k$, $k=0,1,\dots ,L$, as in Step 4,
we now insert an additional copy of $\mathcal{S}_i^{\mathbf{G}}$ whenever $\epsilon_i=1$.
See Figure \ref{fig:partition-split-form-3}.
Thus, we obtain a sequence of subsets $\widehat{\mathcal{S}}_k$ of $\mathcal{S}$, $k=0,1,\dots,\widehat{L}$, where
$$
\widehat{L}+1=L+1+\sum_{i=0}^{\ell'}\epsilon_i=3+2\ell'+\sum_{i=0}^{\ell'}(\ell_i+\epsilon_i),
$$
because we add an additional subset in the sequence whenever $\epsilon_i=1$.
Let
$$
\widehat{L}_i=L_i+\epsilon_0+\dots +\epsilon_{i-1}=
\ell_0+\epsilon_0+\dots +\ell_{i-1}+\epsilon_{i-1}+\ell_{i}+2i+1,
$$
for $0\leq i\leq \ell'$, and set $\widehat{L}_{-1}=-1$ and $\epsilon_{-1}=0$. Then,
$$
\widehat{\mathcal{S}}_{\widehat{L}_i}=\widehat{\mathcal{S}}_{\widehat{L}_i+\epsilon_i}=\mathcal{S}_i^{\mathbf{G}},
$$
for $0\leq i\leq \ell'$,
and
$$
\widehat{\mathcal{S}}_{\widehat{L}_i+\epsilon_i+j}=\widetilde{\mathcal{S}}_{i+1,j-1},
$$
for $-1\leq i\leq\ell'$ and $1\leq j\leq \ell_{i+1}+1$.
Conversely,
$$
\widetilde{\mathcal{S}}_{i,j}=\widehat{\mathcal{S}}_{\widehat{L}_{i-1}+\epsilon_{i-1}+j+1},
$$
for $0\leq i\leq\ell'+1$ and $0\leq j\leq \ell_{i}$.

Recall that the image in $\mathcal{M}$ of the Jacquet--Langlands
correspondence on the set of triples $\mathcal{M}'=\mathcal{M}_{\{P'\},\varphi(\pi')}$ is denoted by $\mathcal{M}^{\mathbf{G}}$ in Step 1.
This image is characterized in terms of segments in Lemma \ref{lem:image-triples}.

\begin{thm}\label{thm:franke-modified}
Let $\mathcal{A}_{\{Q\},\varphi(\sigma)}$ be the space of automorphic forms on $G_{nd}(\mathbb{A})$ with the cuspidal support
in the associate class represented by $\sigma$, obtained as in Step 1 from the cuspidal support $\{P'\},\varphi(\pi')$
of automorphic forms on $G_n'(\mathbb{A})$ represented by $\pi'$. Then, $\mathcal{A}_{\{Q\},\varphi(\sigma)}$ admits a
finite descending filtration
\begin{equation}\label{eq:filt-in-thm-modified}
  \mathcal{A}_{\{Q\},\varphi(\sigma)}=\widehat{\mathcal{A}}_{\{Q\},\varphi(\sigma)}^0\supsetneqq
  \widehat{\mathcal{A}}_{\{Q\},\varphi(\sigma)}^1\supsetneqq \dots \supsetneqq \widehat{\mathcal{A}}_{\{Q\},\varphi(\sigma)}^{\widehat{L}}\supsetneqq \widehat{\mathcal{A}}_{\{Q\},\varphi(\sigma)}^{\widehat{L}+1}=\{0\},
\end{equation}
where the consecutive quotients of the filtration are isomorphic to
\begin{equation}\label{eq:qts-in-thm-modified}
\widehat{\mathcal{A}}_{\{Q\},\varphi(\sigma)}^k\slash \widehat{\mathcal{A}}_{\{Q\},\varphi(\sigma)}^{k+1}\cong
\colim_{\substack{(R,\Pi,\underline{z})\in\mathcal{M} \\ \iota(\underline{z})\in \widehat{S}_{k}}}
\left(
{\Ind }_{R(\mathbb{A})}^{G_{nd}(\mathbb{A})}\left(\Pi\otimes\nu_{\underline{z}}\right)
\otimes S(\check{\mathfrak{a}}_{R,\mathbb{C}}^{G_{nd}})
\right),
\end{equation}
for $0\leq k\leq \widehat{L}$ such that $k\neq\widehat{L}_i,\widehat{L}_i+\epsilon_i$, where
$\widehat{\mathcal{S}}_k$ are the subsets in the sequence introduced above,
and
\begin{equation}\label{eq:qts-in-thm-modified-0}
\widehat{\mathcal{A}}_{\{Q\},\varphi(\sigma)}^{\widehat{L}_i}\slash \widehat{\mathcal{A}}_{\{Q\},\varphi(\sigma)}^{\widehat{L}_i+1}\cong
\colim_{\substack{(R,\Pi,\underline{z})\in\mathcal{M}^{\mathbf{G}} \\ \iota(\underline{z})\in \mathcal{S}_i^{\mathbf{G}}}}
\left(
{\Ind }_{R(\mathbb{A})}^{G_{nd}(\mathbb{A})}\left(\Pi\otimes\nu_{\underline{z}}\right)
\otimes S(\check{\mathfrak{a}}_{R,\mathbb{C}}^{G_{nd}})
\right),
\end{equation}
and if $\epsilon_i=1$, then
\begin{equation}\label{eq:qts-in-thm-modified-1}
\widehat{\mathcal{A}}_{\{Q\},\varphi(\sigma)}^{\widehat{L}_i+1}\slash \widehat{\mathcal{A}}_{\{Q\},\varphi(\sigma)}^{\widehat{L}_i+2}\cong
\colim_{\substack{(S,\Sigma,\underline{\zeta})\not\in\mathcal{M}^{\mathbf{G}} \\ \iota(\underline{\zeta})\in \mathcal{S}_i^{\mathbf{G}}}}
\left(
{\Ind }_{S(\mathbb{A})}^{G_{nd}(\mathbb{A})}\left(\Sigma\otimes\nu_{\underline{\zeta}}\right)
\otimes S(\check{\mathfrak{a}}_{S,\mathbb{C}}^{G_{nd}})
\right).
\end{equation}
\end{thm}

\begin{proof}
The only difference between this filtration and the filtration of Theorem \ref{thm:franke-first-choice}
is that each of the quotients $\mathcal{A}_{\{Q\},\varphi(\sigma)}^{L_i}\slash \mathcal{A}_{\{Q\},\varphi(\sigma)}^{L_i+1}$,
for $i=0,1,\dots ,\ell'$, is divided in two different quotients whenever $\epsilon_i=1$. This is possible if the
two new quotients form a direct sum that gives the original quotient. That is true because we have a direct sum decomposition
\begin{align*}
  \mathcal{A}_{\{Q\},\varphi(\sigma)}^{L_i}\slash \mathcal{A}_{\{Q\},\varphi(\sigma)}^{L_i+1} \cong& \colim_{\substack{(R,\Pi,\underline{z})\in\mathcal{M} \\ \iota(\underline{z})\in \mathcal{S}_{k}}}
\left(
{\Ind }_{R(\mathbb{A})}^{G_{nd}(\mathbb{A})}\left(\Pi\otimes\nu_{\underline{z}}\right)
\otimes S(\check{\mathfrak{a}}_{R,\mathbb{C}}^{G_{nd}})
\right) \\
   \cong&  \colim_{\substack{(R,\Pi,\underline{z})\in\mathcal{M}^{\mathbf{G}} \\ \iota(\underline{z})\in \mathcal{S}_i^{\mathbf{G}}}}
\left(
{\Ind }_{R(\mathbb{A})}^{G_{nd}(\mathbb{A})}\left(\Pi\otimes\nu_{\underline{z}}\right)
\otimes S(\check{\mathfrak{a}}_{R,\mathbb{C}}^{G_{nd}})
\right) \\
& \bigoplus
\colim_{\substack{(S,\Sigma,\underline{\zeta})\not\in\mathcal{M}^{\mathbf{G}} \\ \iota(\underline{\zeta})\in \mathcal{S}_i^{\mathbf{G}}}}
\left(
{\Ind }_{S(\mathbb{A})}^{G_{nd}(\mathbb{A})}\left(\Sigma\otimes\nu_{\underline{\zeta}}\right)
\otimes S(\check{\mathfrak{a}}_{S,\mathbb{C}}^{G_{nd}})
\right).
\end{align*}
More precisely, Lemma \ref{lem:no-morphs-between-transfer-and-non} implies that there are no isomorphisms between triples in $\mathcal{M}^{\mathbf{G}}$
and those not in $\mathcal{M}^{\mathbf{G}}$, so that the colimit reduces to a direct sum, as claimed.
\end{proof}

\subsection*{STEP 7: The main result -- an explicit extension of the global Jacquet--Langlands correspondence}

We are now in position to state the main result of the paper, that is, the global Jacquet--Langlands correspondence
from the Franke filtration on the inner form $G_n'(\mathbb{A})$ introduced in Theorem \ref{thm:franke-inner-form}
to a convenient choice of a functorial refinement of the Franke filtration on $G_{nd}(\mathbb{A})$ obtained in Theorem \ref{thm:franke-modified}.
We use below the notation of these two theorems, but first we finally define the Jacquet--Langlands correspondence
of automorphic forms.

The Jacquet--Langlands correspondence $\mathbf{G}$ is defined in Step 1 on the set of triples $\mathcal{M}'=\mathcal{M}_{\{P'\},\varphi(\pi')}$
with values in the set of triples $\mathcal{M}=\mathcal{M}_{\{Q\},\varphi(\sigma)}$. The sets of triples carry the structure
of a groupoid, and in Step 2 the action of $\mathbf{G}$ on isomorphisms is defined.

In the construction of the Franke filtration, as recalled in Sect.~\ref{sect:filtr-defn},
the consecutive quotients are described in equation \eqref{eq:filt-qts} in terms of induced representations.
The isomorphism is obtained using a functor $M'=M_{\{P'\},\varphi(\pi')}$ (resp.~$M=M_{\{Q\},\varphi(\sigma)}$) from the groupoid $\mathcal{M}'$ (resp.~$\mathcal{M}$) of triples to the category $\mathcal{C}'$ (resp.~$\mathcal{C}$) of automorphic representations of $G_n'(\mathbb{A})$ (resp.~$G_{nd}(\mathbb{A})$), as defined in \eqref{eq:functor}, where the notion of the automorphic representation is in the sense of Sect.~\ref{sect:preliminaries}.
See also \cite{franke:filtration} or \cite[Sect.~2]{grbac-grobner:franke-gln} for the precise definitions.

Using these functors, we define the Jacquet--Langlands correspondence for automorphic forms, denoted again by $\mathbf{G}$,
as the functor defined on the subcategory of $\mathcal{C}'$ obtained as the image of the functor $M'$ such that
the diagram
$$
\xymatrix{
\mathcal{M}' \ar[rr]^{M'}\ar[dd]^{\mathbf G}& &\mathrm{Im}(M')\subseteq \mathcal{C}'\ar[dd]^{\mathbf G}\\
& & \\
\mathcal M\ar[rr]^{M} & & \mathrm{Im}(M)\subseteq \mathcal{C} }
$$
commutes, where the vertical arrow on the left-hand side is the Jacquet--Langlands correspondence for triples,
and horizontal arrows are the functors $M$ and $M'$ from the construction of the Franke filtration.

More precisely, given a triple $(R',\Pi',\underline{z}')$ in $\mathcal{M}'$, the Jacquet--Langlands correspondence is defined as
\begin{align*}
  \mathbf{G}\big(M'\big((R',\Pi',\underline{z}')\big)\big) &= \mathbf{G}\left({\Ind }_{R'(\mathbb{A})}^{G_n'(\mathbb{A})}\left(\Pi'\otimes\nu_{\underline{z}'}\right)
\otimes S(\check{\mathfrak{a}}_{R',\mathbb{C}}^{G_n'})\right)  \\
   &={\Ind }_{R(\mathbb{A})}^{G_{nd}(\mathbb{A})}\left(\Pi\otimes\nu_{\underline{z}}\right)
\otimes S(\check{\mathfrak{a}}_{R,\mathbb{C}}^{G_{nd}}) \\
  &=M \big((R,\Pi,\underline{z})\big)
\end{align*}
where $(R,\Pi,\underline{z})=\mathbf{G}\big((R',\Pi',\underline{z}')\big)$ is the triple in $\mathcal{M}$, as in Step 1.
The Jacquet--Langlands correspondence for triples relates the isomorphisms given by the Weyl group elements identified with the same permutation,
as explained in Step 2. The functors $M'$ and $M$ send an isomorphism of triples given by a Weyl group element to the intertwining
operator acting on the induced representation associated with the same Weyl group element, cf.~Sect.~\ref{sect:filtr-defn}.
Hence, the Jacquet--Langlands correspondence $\mathbf{G}$
for automorphic forms sends the intertwining operator associated with the Weyl group element $w'$
to the intertwining operator associated with the corresponding Weyl group element $w$. Thus, the Jacquet--Langlands
correspondence preserves intertwining operators between induced representations.
In particular, the Jacquet--Langlands correspondence $\mathbf{G}$ for automorphic forms preserves the structure of diagrams over which the colimits are taken
in the construction of the Franke filtration,
so that $\mathbf{G}$ commutes with formation of colimits. Therefore, all the colimits, including direct sums as a special case, are preserved under the
Jacquet--Langlands correspondence for automorphic forms.

\begin{thm}\label{thm:main-result-JL-beyond}
Let $\mathcal{A}_{\{P'\},\varphi(\pi')}$ be the space of automorphic forms on $G_n'(\mathbb{A})$
with cuspidal support in the associate class $\varphi(\pi')$.
Let $\mathcal{A}_{\{Q\},\varphi(\sigma)}=\mathbf{G}\left(\mathcal{A}_{\{P'\},\varphi(\pi')}\right)$
be the space of automorphic forms on $G_{nd}(\mathbb{A})$ corresponding to $\mathcal{A}_{\{P'\},\varphi(\pi')}$
in the Jacquet--Langlands correspondence as in Step 0.
Let
$$
\mathcal{A}_{\{P'\},\varphi(\pi')}=\mathcal{A}_{\{P'\},\varphi(\pi')}^0\supsetneqq
\mathcal{A}_{\{P'\},\varphi(\pi')}^1\supsetneqq \dots \supsetneqq \mathcal{A}_{\{P'\},\varphi(\pi')}^{\ell'}\supsetneqq \mathcal{A}_{\{P'\},\varphi(\pi')}^{\ell'+1}=\{0\},
$$
be the Franke filtration of the space $\mathcal{A}_{\{P'\},\varphi(\pi')}$ as in Theorem \ref{thm:franke-inner-form}.
Then, there exists a filtration
$$
\mathcal{A}_{\{Q\},\varphi(\sigma)}=\widehat{\mathcal{A}}_{\{Q\},\varphi(\sigma)}^0\supsetneqq
\widehat{\mathcal{A}}_{\{Q\},\varphi(\sigma)}^1\supsetneqq \dots \supsetneqq \widehat{\mathcal{A}}_{\{Q\},\varphi(\sigma)}^{\widehat{L}}\supsetneqq \widehat{\mathcal{A}}_{\{Q\},\varphi(\sigma)}^{\widehat{L}+1}=\{0\},
$$
of the space $\mathcal{A}_{\{Q\},\varphi(\sigma)}$, which is obtained in Theorem \ref{thm:franke-modified} as a functorial refinement of the filtration defined by Franke,
such that the global Jacquet--Langlands correspondence of spaces of automorphic forms satisfies
$$
\mathbf{G}\left(\mathcal{A}_{\{P'\},\varphi(\pi')}^i\slash \mathcal{A}_{\{P'\},\varphi(\pi')}^{i+1}\right)\cong
\mathcal{A}_{\{Q\},\varphi(\sigma)}^{\widehat{L}_i}\slash \mathcal{A}_{\{Q\},\varphi(\sigma)}^{\widehat{L}_i+1},
$$
for all $0\leq i\leq\ell'$, and all the remaining quotients in the filtration of $\mathcal{A}_{\{Q\},\varphi(\sigma)}$ are not in the image of $\mathbf{G}$
\end{thm}

\begin{proof}
The theorem is now a consequence of all the preparatory work. According to the basic properties of the Jacquet--Langlands correspondence $\mathbf{G}$ for automorphic forms
defined above, we calculate the image of consecutive quotients of the Franke filtration of $\mathcal{A}_{\{P'\},\varphi(\pi')}$ as follows
\begin{align*}
  \mathbf{G}\left(\mathcal{A}_{\{P'\},\varphi(\pi')}^i\slash \mathcal{A}_{\{P'\},\varphi(\pi')}^{i+1}\right) &=
\mathbf{G}\left(\colim_{\substack{(R',\Pi',\underline{z}')\in\mathcal{M}' \\ \iota'(\underline{z}')\in \mathcal{S}'_{i}}}
\left(
{\Ind }_{R'(\mathbb{A})}^{G_n'(\mathbb{A})}\left(\Pi'\otimes\nu_{\underline{z}'}\right)
\otimes S(\check{\mathfrak{a}}_{R',\mathbb{C}}^{G_n'})
\right)\right)
\\
   &=  \colim_{\substack{(R,\Pi,\underline{z})\in\mathcal{M}^{\mathbf{G}} \\ \iota(\underline{z})\in \mathcal{S}_{i}^{\mathbf{G}}}}
\left(
{\Ind }_{R(\mathbb{A})}^{G_{nd}(\mathbb{A})}\left(\Pi\otimes\nu_{\underline{z}}\right)
\otimes S(\check{\mathfrak{a}}_{R,\mathbb{C}}^{G_{nd}})
\right),
\end{align*}
where $(R,\Pi,\underline{z})=\mathbf{G}\big((R',\Pi',\underline{z}')\big)$ is obviously in the image $\mathcal{M}^{\mathbf{G}}$,
$\iota(\underline{z})$ is in $\mathcal{S}_i^{\mathbf{G}}$ by definition of $\mathcal{S}_i^{\mathbf{G}}$ in Step 4.
We also used the fact that $\mathbf{G}$ commutes with colimits. Comparing with the description of the quotient
$$
\mathcal{A}_{\{Q\},\varphi(\sigma)}^{\widehat{L}_i}\slash \mathcal{A}_{\{Q\},\varphi(\sigma)}^{\widehat{L}_i+1}
$$
in Theorem \ref{thm:franke-modified}, we see that the first claim of the theorem holds. The remaining quotients of the filtration
of $\mathcal{A}_{\{Q\},\varphi(\sigma)}$ are clearly not in the image, because there are no quotients left in the
Franke filtration of $\mathcal{A}_{\{P'\},\varphi(\pi')}$.
\end{proof}

\end{document}